\documentclass[9pt]{article}
\usepackage[top=1.0in, right=1.0in, left=1.0in, bottom=1.0in]{geometry}
\usepackage{authblk}

\usepackage{wrapfig}

\usepackage{cite}
\usepackage{algorithmic}
\usepackage{graphicx}
\usepackage{textcomp}
\def\BibTeX{{\rm B\kern-.05em{\sc i\kern-.025em b}\kern-.08em
    T\kern-.1667em\lower.7ex\hbox{E}\kern-.125emX}}
\usepackage{array}
\usepackage{tabu}
\usepackage{amsmath}
\usepackage{amsthm}
\usepackage{amssymb}
\usepackage{mathtools}
\usepackage[utf8]{inputenc} 
\usepackage[T1]{fontenc}    
\usepackage{hyperref}       
\usepackage{url}            
\usepackage{booktabs}       
\usepackage{amsfonts}       
\usepackage{nicefrac}       
\usepackage{color}
\usepackage[scriptsize]{subfigure}
\usepackage{hyperref}

\definecolor{subsectioncolor}{RGB}{1,1,0}
\usepackage{tikz}

\usetikzlibrary{shapes,arrows}
\usetikzlibrary{arrows,calc,positioning}
\usetikzlibrary{arrows.meta, positioning, shadows}

\tikzstyle{block} = [draw,rectangle,thick,minimum height=2em,minimum width=2em]
\tikzstyle{sum} = [draw,circle,inner sep=0mm,minimum size=2mm]
\tikzstyle{connector} = [->,thick]
\tikzstyle{line} = [thick]
\tikzstyle{branch} = [circle,inner sep=0pt,minimum size=1mm,fill=black,draw=black]
\tikzstyle{guide} = []
\tikzstyle{snakeline} = [connector, decorate, decoration={pre length=0.2cm,
                         post length=0.2cm, snake, amplitude=.4mm,
                         segment length=2mm},thick, magenta, ->]

\graphicspath{{epsfiles/}}

\newcommand{\longthmtitle}[1]{\mbox{} \emph{\textbf{(#1)}}}

\newcommand{\new}[1]{{\color{black} #1}}

\newcommand{\Ac}{\mathcal{A}}

\newcommand{\Cc}{\mathcal{C}}

\newcommand{\Ec}{\mathcal{E}}
\newcommand{\Fc}{\mathcal{F}}
\newcommand{\Gc}{\mathcal{G}}
\newcommand{\Hc}{\mathcal{H}}

\newcommand{\Kc}{\mathcal{K}}
\newcommand{\Lc}{\mathcal{L}}

\newcommand{\Nc}{\mathcal{N}}

\newcommand{\Uc}{\mathcal{U}}
\newcommand{\Vc}{\mathcal{V}}

\newcommand{\Zc}{\mathcal{Z}}
\newcommand{\abs}[1]{\left| #1 \right|}
\newcommand{\sgn}{\textnormal{sgn}}
\newcommand{\dist}{\textnormal{dist}}
\newcommand{\VI}{\textnormal{VI}}
\newcommand{\SOL}{\textnormal{SOL}}

\newcommand{\oprocendsymbol}{\hbox{$\bullet$}}
\newcommand{\oprocend}{\relax\ifmmode\else\unskip\hfill\fi\oprocendsymbol}

\newcommand{\real}{\ensuremath{\mathbb{R}}}

\newcommand{\until}[1]{\{1,\dots,#1\}}
\newcommand{\ontil}[1]{\{0,\dots,#1\}}

\newcommand{\norm}[1]{\left\lVert#1\right\rVert}

\newcommand{\Dini}{D^{+}}

\newcommand{\pfrac}[2]{\frac{\partial #1}{\partial #2}}

\newcommand{\lambdamin}{\lambda_\textnormal{min}}
\newcommand{\lambdamax}{\lambda_\textnormal{max}}
\newcommand{\proj}[2]{\textnormal{proj}_{#1}\left(#2\right)}

\newtheorem{theorem}{Theorem}[section]
\newtheorem{proposition}[theorem]{Proposition}
\newtheorem{lemma}[theorem]{Lemma}

\newtheorem{remark}[theorem]{Remark}

\newtheorem{problem}{Problem}

\newtheorem{assumption}{Assumption}

\allowdisplaybreaks

\title{Anytime Solvers for Variational Inequalities: \\ the (Recursive) Safe Monotone Flows}
\date{}
\author{Ahmed Allibhoy and Jorge Cort\'es}

\begin{document}
\maketitle

\begin{abstract}
  This paper synthesizes anytime algorithms, in the form of
  continuous-time dynamical systems, to solve monotone variational
  inequalities. We introduce three algorithms that solve this
  problem: the projected monotone flow, the safe monotone flow, and
  the recursive safe monotone flow.  The first two systems admit
  dual interpretations: either as projected dynamical systems or as
  dynamical systems controlled with a feedback controller
  specified by a quadratic program
  derived using techniques from safety-critical control.  The
  third flow bypasses the need to solve quadratic programs along the
  trajectories by incorporating a dynamics whose equilibria
  precisely correspond to such solutions, and interconnecting the
  dynamical systems on different time scales.  We perform a thorough
  analysis of the dynamical properties of all three systems. For the
  safe monotone flow, we show that equilibria correspond exactly
  with critical points of the original problem, and the constraint
  set is forward invariant and asymptotically stable.  The
  additional assumption of convexity and monotonicity allows us to
  derive global stability guarantees, as well as establish the
  system is contracting when the constraint set is polyhedral.  For
  the recursive safe monotone flow, we use tools from singular
  perturbation theory for contracting systems to show KKT points are
  locally exponentially stable and globally attracting, and obtain
  practical safety guarantees. We illustrate the performance of the
  flows on a two-player game example and also demonstrate the
  versatility for interconnection and regulation of dynamical
  processes of the safe monotone flow in an example of a receding
  horizon linear quadratic dynamic game.
\end{abstract}


\section{Introduction}
Variational inequalities encompass a wide range of problems arising in
operations research and engineering applications, including minimizing
a function, characterizing Nash equilibria of a game, and seeking
saddle points of a function.  This paper synthesizes continuous-time
flows whose \emph{trajectories converge to the solution} set of a
monotone variational inequality while \emph{respecting the constraints
  at all times}: hence the terminology \emph{anytime solvers}.

Our motivation for considering this problem is two-fold. First,
iterative algorithms in numerical computing can be interpreted as
dynamical systems. This opens the door for the use of controls and
system-theoretic tools to characterize their qualitative and
quantitative properties, e.g., stability of solutions, convergence
rate, and robustness to disturbances.  In turn, the availability of
such characterization sets the stage for developing sample-data
implementations and systematically designing new algorithms equipped
with desired properties.

The second motivation stems from problems where the solution to the
variational inequality is used to regulate a physical process modeled
as a dynamically evolving plant (e.g., providing setpoints, specifying
optimization-based controllers, steering the plant toward an optimal
steady-state). This type of problem arises in multiple application
areas, including power systems, network congestion control, and
traffic networks.  In these settings, the algorithm used to solve the
variational inequality is interconnected with a plant, and thus the
resulting coupled system naturally lends itself to system-theoretic
analysis and design.  We are particularly interested in situations
where the problem incorporates constraints which, when violated, would
threaten the safe operation of the physical system. In such cases, it
is desirable that the algorithm is \emph{anytime}, meaning that it is
guaranteed to return a feasible point even when terminated before it
has converged to a solution. The anytime property ensures that the
specifications conveyed to the plant remain feasible at all times.

\subsection{Related Work}
The study of the interplay between continuous-time dynamical systems
and variational inequalities has a rich history. In the context of
optimization, classical references
include~\new{\cite{KA-LH-HU:58,RWB:91,UH-JBM:94,CH:73}}.  For general
variational inequalities, flows solving them can be obtained through
differential inclusions involving monotone set-valued maps, originally
introduced in \cite{HB:73}. These systems have been equivalently
described as projected dynamical systems~\cite{AN-DZ:96} and
complementarity systems
\cite{BB-AD-CL-VA:06,WPMHH-JMS-SW:00}. Recently, this framework has
been extended to settings where monotonicity holds with respect to
non-Euclidean norms~\cite{WPMHH-MKC-MFH:20}, and when the system
evolves on a Riemannian manifold~\cite{AH-SB-FD:21}.  One limitation
of these systems is that they are, in general, discontinuous, which
poses challenges both for their theoretical analysis and practical
implementation.

The interconnection of algorithms with physical plants has attracted
much attention recently. The interconnected system can often be viewed
as the coupling of a dynamical system with a set valued operator
\new{(cf. \cite{AT-BB-CP:18, BB-AT:20} for a comprehensive survey of systems having
  this form, and control design problems within this framework)}.  The specific setting where the algorithm optimizes the
steady-state of the plant is typically referred to as \emph{online
  feedback optimization}\cite{MC-EDA-AB:20, AH-ZH-SB-GH-FD:24} and has
been studied in continuous time in the context of power
systems~\cite{ED-AS:16,LSPL-JWSP-EM:21}, network congestion
control~\cite{SHL-FP-JCD:02}, and traffic
networks~\cite{GB-JC-JIP-EDA:22-tcns}, as well as discrete time
in~\cite{VH-AH-LO-SB-FD:21}.  Recently this framework has also been
studied in game scenarios~\cite{AA-JWSP-LP:22,
  GB-DMPL-MHDB-SB-RSS-JL-FD:24}.

A complementary approach uses extremum seeking
control~\cite{KBA-MK:03}, which has been generalized to the setting of
games in~\cite{PF-MK-TB:12}.  Extremum seeking differs from the
methods introduced here in that they are typically zeroth-order
methods, and do not offer exact stability guarantees. The recent
work~\cite{AW-MK-AS:23} considers safety guarantees for extremum
seeking control for the special case of a single inequality
constraint. In fact, the proposed algorithm can be understood as an
approximation of the safe gradient flow in~\cite{AA-JC:24-tac}, which
is a precursor for constrained optimization of the algorithms proposed
here for constraint sets parameterized by multiple inequality and
equalities.

To synthesize our flows, we employ techniques from safety-critical
control~\cite{ADA-SC-ME-GN-KS-PT:19,MC-CB:23}, which refers to the
problem of designing a feedback controller that ensures that the state
of the system satisfies certain constraints.
The problem of ensuring safety is typically formalized by specifying a
set of states where the system is said to remain safe, and ensuring
the safe set is forward invariant.  The work~\cite{FB:99} reviews set
invariance in control. A popular technique for synthesizing safe
controllers uses the concept of control barrier functions (CBFs),
see~\cite{PW-FA:07,ADA-XX-JWG-PT:17,WX-CGC-CB:23} and
references~therein, to specify optimization-based feedback controllers
which ``filter'' a nominal system to ensure it remains in the safe
set.  Here, we apply this strategy to synthesize anytime algorithms,
viewing the constraint set as a safety set and the monotone operator
of the variational inequality as the nominal system. This view has
connections to projected dynamical systems, whose relationship with
CBF-based control design has recently been explored
in~\cite{GD-WPMHH:23,GD-JC-WPMHH:24}, and leads to the alternative
``projection-based'' interpretation of the projected monotone flow and
safe monotone flow proposed here.

\subsection{Statement of Contributions} We consider the synthesis of
continuous-time dynamical systems solving variational inequalities
while respecting the constraints at all times.
\new{We work primarily with continuous-time systems since we are
  motivated by online feedback optimization, where the evolution of a
  physical process determines the functions that define the
  variational inequality, which in turn regulates the dynamically
  evolving plant.  The methods introduced here can be implemented on
  digital computers by discretizing the resulting dynamics.}

We discuss three flows that solve this problem.  The \emph{projected
  monotone flow} is already known, but we provide a novel
control-theoretic interpretation as a control system whose closed-loop
behavior is as close as possible to the monotone operator while still
belonging to the tangent cone of the constraint set.  The \emph{safe
  monotone flow} can be interpreted either as a control system with a
feedback controller synthesized using techniques from safety-critical
control or as an approximation of the projected monotone flow.  The
latter interpretation relies on the novel notion of restricted tangent
set, which generalizes the usual concept of tangent cone from
variational geometry.  We show that equilibria correspond exactly with
critical points of the original problem, establish existence and
uniqueness of solutions, and characterize the regularity properties of
flow.  We also show that the constraint set is forward invariant and
asymptotically stable, and derive global stability guarantees under
the additional assumption of convexity and monotonicity. Our technical
analysis relies on a suite of Lyapunov functions to establish
stability properties with respect to the constraint set and the whole
state space.  When the constraint set is polyhedral, we establish that
the system is contracting and exponentially stable.  Finally, the
\emph{recursive safe monotone flow} bypasses the need for continuously
solving quadratic programs along the trajectories by incorporating a
dynamics whose equilibria precisely correspond to such solutions, and
interconnecting the dynamical systems on different time scales.  Using
tools from singular perturbation theory for contracting systems, we
show that for variational inequalities with polyhedral constraints,
the KKT points are locally exponentially stable and globally
attracting, and obtain practical safety guarantees.  We compare the
three flows on a simple two-player game and demonstrate how the safe
monotone flow can be interconnected with dynamical processes on an
example of a receding horizon linear quadratic dynamic game.

The \new{flows} introduced here generalize the safe gradient flow, a
continuous-time system proposed in our previous
work~\cite{AA-JC:24-tac} (in parallel,~\cite{VH-AH-LO-SB-FD:21}
introduced a discrete-time implementation of a simplified version of
it).  With respect to the safe gradient flow, our treatment extends
the results in three key ways. First, we consider variational
inequalities, rather than just constrained optimization problems,
making the flows introduced here applicable to a much broader range of
problems. Second, with assumptions of monotonicity and convexity of
the constraints, we obtain global stability and convergence results,
rather than local stability results.  Third, the rigorous
characterization of the contractivity properties of the safe monotone
flow paves the way for its interconnection with other dynamically
evolving processes. In fact, the proposed recursive safe monotone flow
critically builds on this analysis by leveraging different timescales
and singular perturbation theory.

\section{Preliminaries}\label{sec:preliminaries}
We review here basic notions from variational inequalities,
projections, and set invariance. Readers familiar with these concepts
can safely skip this section.

\subsection{Notation}
Let $\real$ denote the set of real numbers. For $c \in \real$,
$[c]_+ = \max\{0, c\}$. For $x \in \real^n$, $x_i$ denotes the $i$th
component and $x_{-i}$ denotes all components of $x$ excluding
$i$. For $v, w \in \real^n$, $v \leq w$ (resp. $v < w$) denotes
$v_i \leq w_i$ (resp. $v_i < w_i$) for $i \in \until{n}$. We let
$\norm{v}$ denote the Euclidean norm, \new{and
  $\norm{v}_\infty := \max_{1 \leq i \leq n} |v_i|$ denote the
  infinity norm}. We write $Q \succeq 0$ (resp.,
$Q \succ 0$) to denote $Q$ is positive semidefinite (resp., $Q$ is
positive definite). Given $Q \succeq 0$, let $\norm{x}_Q$ denote the
seminorm where $\norm{x}_Q = \sqrt{x^\top Qx}$.  For a symmetric $Q$,
$\lambda_\text{min}(Q)$ and $\lambdamax{(Q)}$ denote the minimum and
maximum eigenvalues of $Q$, resp. Given $\Cc \subset \real^n$, the
distance of $x \in \real^n$ to $\Cc$ is
$\dist(x, \Cc) = \inf_{y \in \Cc} \norm{x - y}$.
The \emph{projection map onto $\overline{\Cc}$} is
$\text{proj}_\Cc:\real^n \rightrightarrows \overline{\Cc}$, where
$\proj{\Cc}{x} = \big\{ y \in \overline{\Cc} \mid \norm{x - y} = \dist(x, \Cc) \big\}$.
Given a closed and convex set
$\Cc \subset \real^n$, the \emph{normal cone} to $\Cc$ at
$x \in \real^n$ is
$N_\Cc(x) = \{d \in \real^n \mid d^\top (x' - x) \leq 0, \;\forall x'
\in \Cc \}$
and the \emph{tangent cone} to $\Cc$ at $x$ is
$T_\Cc(x) = \{ \xi \in \real^n \mid d^\top \xi \leq 0, \;\forall d \in
N_\Cc(x) \}$.

Given $g:\real^n \to \real$, we denote its gradient by $\nabla g$. For
$g:\real^n \to \real^m$, $\pfrac{g(x)}{x}$ denotes its Jacobian.  For
$I \subset \{1, 2, \dots, m \}$, we denote by $\pfrac{g_{I}(x)}{x}$
the matrix whose rows are $\{\nabla g_i(x)^\top\}_{i \in I}$. Given 
a function $f:\real^n \to \real$,
we say it is directionally differentiable if for all $\xi \in \real^n$, 
the following limit exists 
\[  f'(x; \xi) = \lim_{h \to 0^+} \frac{f(x + h\xi) - f(x)}{h}. \]
Given a
vector field $\Gc:\real^n \to \real^n$ and a function
$V:\real^n \to \real$, the \emph{Upper-right Dini derivative} of $V$
along $\Gc$ is 
\[
  \Dini_{\Gc}V(x) = \limsup_{h \to 0^+} \frac{1}{h}\left( V(\Phi_h(x)) -
    V(x) \right),
\]
where $\Phi_h$ is the flow map\footnote{\new{For
    $h \in \real_{\geq 0}$, the flow map $\Phi_h: \real^n \to \real^n$
    is defined by $\Phi_h(x_0) = x(h)$, where $x(t)$ is the unique
    trajectory solving $\dot{x} = \Gc(x)$ with $x(0) = x_0$.}}
corresponding to the system $\dot{x} = \Gc(x)$.  When $V$ is directionally
differentiable $\Dini_{\Gc}V(x) = V'(x, \Gc(x))$ and when $V$ is
differentiable, then $\Dini_{\Gc}V(x)=\nabla V(x)^\top \Gc(x)$.

\subsection{Variational Inequalities}\label{sec:var-ineqs}
Here we review the basic theory of variational inequalities
following~\cite{FF-JSP:03}. Let $F:\real^n \to \real^n$ be a map and
$\Cc \subset \real^n$ a set of constraints.  A variational inequality
refers to the problem of finding $x^* \in \Cc$ such that
\begin{equation}
  \label{eq:variational-inequality}
  (x - x^*)^\top F(x^*) \geq 0, \qquad \forall x \in \Cc.
\end{equation}
We use $\VI(F, \Cc)$ to refer to the
problem~\eqref{eq:variational-inequality} and $\SOL(F, \Cc)$ to denote
its set of solutions. Variational inequalities provide a framework to
study many different analysis and optimization problems, including
\begin{itemize}
\item Solving the nonlinear equation $F(x^*) = 0$, which corresponds
  to $\VI(F, \real^n)$;
\item Minimizing the function $f:\real^n \to \real$ subject to the
  constraint that $x \in \Cc$, which corresponds to
  $\VI(\nabla f, \Cc)$;
\item Finding saddle points of the function
  $\ell:\real^n \times \real^m \to \real$ subject to the constraints
  that $x_1 \in X_1$ and $x_2 \in X_2$, which corresponds to
  $\VI([\nabla_{x_1} \ell ; -\nabla_{x_2} \ell], X_1 \times X_2)$.
\item Finding the Nash equilibria of a game with $N$ agents, where the
  $i$th agent wants to minimize the cost $J_i(x_i, x_{-i})$ subject to
  the constraint $x_i \in X_i$, which corresponds to $\VI(F, \Cc)$,
  where $F$ is the \emph{pseudogradient} operator defined by
  $ F(x) = (\nabla_{x_1}J_1(x), \dots, \nabla_{x_N}J_N(x))$,
  and $\Cc = X_1 \times X_2 \times \cdots \times X_N$.
\end{itemize}
The map $F:\real^n \to \real^n$ is \emph{monotone} if
$ (x_1 - x_2)^\top (F(x_1) - F(x_2)) \geq 0$, for all
$x_1, x_2 \in \real^n$, and $F$ is \emph{$\mu$-strongly monotone} if
there exists $\mu > 0$ such that
$ (x_1 - x_2)^\top (F(x_1) - F(x_2)) \geq \mu \norm{x_1 - x_2}^2$, for
all $x_1, x_2 \in \real^n$. When $F$ is a gradient map, i.e.,
$F = \nabla f$ for some function $f:\real^n \to \real$, then
monotonicity (resp. $\mu$-strong monotonicity) is equivalent to
convexity (resp. strong convexity) of $f$. When $F$ is monotone and
$\Cc$ is convex, $\VI(F, \Cc)$ is a \emph{monotone variational
  inequality}.

In order to provide a characterization of the solution set
$\SOL(F, \Cc)$, we need to introduce a more explicit description of
the set of constraints.  Suppose that $g:\real^n \to \real^m$ and
$h:\real^n \to \real^k$ are continuously differentiable and $\Cc$ is
described by inequality constraints and affine equality
constraints,
\begin{equation}
  \label{eq:constraint-set}
  \Cc = \{ x \in \real^n \mid g(x) \leq 0,\; h(x) = Hx - c_{h} = 0 \},
\end{equation}
where $H \in \real^{k \times n}$ and $c_{h} \in \real^{k}$.  For
$x \in \real^n$, we denote the \emph{active constraint}, \emph{inactive
  constraint}, and \emph{constraint violation} sets, resp., as
\[
  \begin{aligned}
    I_0(x) &= \{ i \in [1, m] \mid g_i(x) = 0 \},
    \\
    I_-(x) &= \{ i \in [1, m] \mid g_i(x) < 0 \},
    \\
    I_+(x) &= \{ i \in [1, m] \mid g_i(x) > 0 \}.
  \end{aligned}
\]
The problem~\eqref{eq:variational-inequality} satisfies the
constant-rank condition at $x \in \Cc$ if there exists an open set $U$
containing $x$ such that for all $I \subset I_0(x)$, the rank of
$\{\nabla g_i(y) \}_{i \in I} \cup \{ \nabla h_j(y) \}_{j=1}^{k}$
remains constant for all $y \in U$.  The
problem~\eqref{eq:variational-inequality} satisfies the
Mangasarian-Fromovitz Constraint Qualification (MFCQ) condition at
$x \in \Cc$ if $\{\nabla h_j(x)\}_{i=1}^{k}$ are linearly independent,
and there exists $\xi \in \real^n$ such that
$\nabla g_i(x)^\top \xi < 0$ for all $i \in I_0(x)$, and
$\nabla h_j(x)^\top \xi = 0$ for all $j \in \until{k}$.  If MFCQ holds
at $x^* \in \Cc$, then if $x^*$ satisfies
\eqref{eq:variational-inequality}, there exists
$(u^*, v^*) \in \real^m \times \real^k$ such that
\begin{subequations}
  \label{eq:KKT}
  \begin{align}
    \label{eq:KKT1}
    F(x^*) + \sum_{i=1}^{m}u^*_i\nabla g_i(x^*) +
    \sum_{j=1}^{k}v_j^*\nabla h_j(x^*) &= 0
    \\ 
    g(x^*) &\leq 0
    \\
    h(x^*) &= 0
    \\
    u^* &\geq 0
    \\
    (u^*)^\top g(x^*) &= 0.
  \end{align}
\end{subequations}
Equations \eqref{eq:KKT} are called the \emph{KKT conditions}. A point
$(x^*, u^*, v^*)$ satisfying them is a \emph{KKT triple} and the pair
$(u^*, v^*)$ is a \emph{Lagrange multiplier}.  For monotone
variational inequalities, when MFCQ holds at $x^*$, then the KKT
conditions are both necessary and sufficient for
$x^* \in \SOL(F, \Cc)$, \new{cf.~\cite[Proposition
  1.3.4]{FF-JSP:03}}. When~$F$ is monotone, $\SOL(F, \Cc)$ is closed
and convex, \new{cf.~\cite[Proposition 2.3.5]{FF-JSP:03}}.  If $F$ is
additionally $\mu$-strongly monotone, and $\SOL(F, \Cc) \neq \emptyset$, then the set of solutions is a
singleton, \new{cf.~\cite[Theorem 2.3.3]{FF-JSP:03}}.

\subsection{Controller Synthesis for Set
  Invariance}\label{sec:set-invariance}
We review here notions from the theory of set invariance for control
systems following~\cite{FB:99} and discuss methods for synthesizing
feedback controllers that ensure invariance.  Consider a control-affine system
\begin{equation}
  \begin{aligned}
    \label{eq:control-system}
    \dot{x} &= \Fc(x, \nu)  \\
    &= F_0(x) + \sum_{i=1}^{r}\nu_iF_i(x),
  \end{aligned}
\end{equation}
with Lipschitz-continuous vector fields $F_i:\real^n \to \real^n$, for
$i \in \ontil{r}$, and a set \new{$\Uc \subset \real^r$} of valid control
inputs~$\nu$.  Let $\Cc \subset \real^n$ be a constraint set of the
form~\eqref{eq:constraint-set} to which we want to restrict the
evolution of the system.  We consider the problem of designing a
feedback controller $\kappa:\real^n \to \Uc$ such that $\Cc$ is forward
invariant with respect to the closed-loop
dynamics $\dot{x} = \Fc(x, \kappa(x))$. In applications, $\Cc$ often
corresponds to the set of states for which the system can operate
safely.  For this reason, we refer to $\Cc$ as the \emph{safety set},
and call the system \emph{safe} under a controller $\kappa$ if $\Cc$ is
forward invariant. A controller ensuring
safety is \emph{safeguarding}.  We discuss two optimization-based
strategies for synthesizing safeguarding controllers.

\subsubsection{Safeguarding Control via
  Projection}\label{sec:proj-feedback}
The first strategy ensures the closed-loop dynamics lie in the tangent
cone of the safety set. Whenever MFCQ holds at $x \in \Cc$, the tangent cone
can conveniently be expressed as, cf. \cite[Theorem
6.31]{RTR-RJBW:98},
$ T_\Cc(x) = \{ \xi \in \real^n \left| \pfrac{h(x)}{x}\xi = 0,
  \pfrac{g_{I_0}(x)}{x}\xi \leq 0\right. \} $.
We then define the
set-valued map $K_\text{proj}:\real^n \rightrightarrows \Uc$ which
characterizes the set of inputs, \new{$\nu$, such 
that $\Fc(x, \nu) \in T_\Cc(x)$.} The set has the form,
\begin{align*}
  K_\text{proj}(x)
    \!=\! \Big\{ \mu \in \Uc \;\Big|\;
  &\Dini_{F_0}g_i(x) \! + \!
    \sum_{\ell=1}^{r}\mu_\ell\Dini_{F_\ell}g_i(x)
    \leq 0,\, i \in I(x),
  \\
  &\Dini_{F_0}h_j(x) +
    \sum_{\ell=1}^{r}\mu_\ell\Dini_{F_\ell}h_j(x)
    = 0,\, j=1, \dots, k  \Big\}.
\end{align*}
As we show in the following result, the feedback $k:\Cc \to \Uc$ such that
$k(x) \in K_\text{proj}(x)$ for $x \in \Cc$ renders $\Cc$ forward
invariant. The proof is omitted for brevity, though it follows
directly from Nagumo's Theorem \cite[Theorem~2]{JPA-AC:84}.

\begin{lemma}[Projection-based Safeguarding
    Feedback]\label{lem:proj-feedback}
    Consider the system \eqref{eq:control-system} with safety set $\Cc$
    and suppose that $K_\text{proj}(x) \neq \emptyset$ for all
    $x \in \Cc$. If $\kappa:\Cc \to \Uc$ is a feedback controller
      such that (i) $\kappa(x) \in K_\text{proj}(x)$ for all
      $x \in \Cc$ and (ii) for all initial conditions on $\Cc$, the closed-loop system
      $\dot{x} = \Fc(x, \kappa(x))$ admits a unique absolutely continuous 
      solution satisfying the dynamics almost everywhere, 
      then $\kappa$ is safeguarding. 
\end{lemma}

To synthesize a safeguarding controller, we propose a strategy where
for each $x \in \Cc$, $\kappa(x)$ is expressed as the solution to a
mathematical program. Because $K_\text{proj}(x)$ is defined in
terms of affine constraints on the control input~$\nu$, we can readily
express a feedback satisfying the hypotheses of
Lemma~\ref{lem:proj-feedback} in the form of a mathematical program,
\begin{equation}
  \label{eq:qp-proj-controller}
  \kappa(x) \in \underset{\nu \in K_\text{proj}(x)}{\text{argmin}} 
  J(x, \nu) ,
\end{equation}
for an appropriate choice of cost function
$J:\Cc \times \Uc \to \real$.  In general, care must be taken to
ensure that the set $K_\text{proj}$ is nonempty and that the
controller $\kappa$ in~\eqref{eq:qp-proj-controller} satisfies appropriate
regularity conditions to ensure existence and uniqueness for solutions
of the resulting closed-loop dynamics.  Even if these properties hold,
the approach has several limitations: the controller is ill-defined
for initial conditions lying outside the safety set and the
closed-loop system in general is nonsmooth.

\subsubsection{Safeguarding Control via Control Barrier
  Functions}\label{sec:cbf-feedback}
The second strategy for synthesizing safeguarding controllers
addresses the limitations of projection-based methods.  The approach
relies on the notion of a vector control barrier
functions~\cite{ADA-SC-ME-GN-KS-PT:19,AA-JC:24-tac}.  Given a set
$X \supset \Cc$ and a set of valid control inputs
$\Uc \subset \real^{m}$, we say the pair
$(g, h):\real^n \times \real^k \to \real^m$ is a \emph{$(m, k)$-vector
  control barrier function} (VCBF) for $\Cc$ on $X$ relative to $\Uc$
if there exists $\alpha > 0$ such that the map
$K_{\textnormal{cbf}, \alpha}:\real^n \rightrightarrows \Uc$ given by
\begin{align*}
  K_{\textnormal{cbf}, \alpha}(x)
  \!=\!
    \Big\{ \mu \in \Uc \;\Big|\;
  &\Dini_{F_0}g_i(x) \!+\!
    \sum_{\ell=1}^{r}\mu_\ell\Dini_{F_\ell}g_i(x) \!+\! \alpha 
    g_i(x) \leq 0,
  \\
  &\Dini_{F_0}h_j(x) + \sum_{\ell=1}^{r}\mu_\ell\Dini_{F_\ell}h_j(x) + \alpha
    h_j(x) = 0,\, 1 \leq i \leq m, \; 1 \le j \leq k \Big\},
\end{align*}
takes nonempty values for all $x \in X$. Similar to the previous
strategy, the set $K_{\textnormal{cbf}, \alpha}$ characterizes the set of
inputs which ensure that the state remains inside the safe set. Unlike
the previous strategy, this assurance is implemented gradually: the
parameter $\alpha$ corresponds to \new{ how tolerant we are of
  increasing the value of inactive constraints, with small values of
  $\alpha$ corresponding to \new{more restrictions on the growth rate}. To see this, note that if $g_i(x) < 0$ then the left-hand
  sides of the constraints parameterizing $K_{\textnormal{cbf}, \alpha}$
  decrease as $\alpha$ increases, making the constraint easier to
  satisfy.} For $\alpha = \infty$, \new{if $g_i(x) < 0$, then the
  left-hand side of the corresponding constraint becomes $-\infty$,
  thus the only nontrivial constraints correspond to $i$ where
  $g_i(x) = 0$}, and \new{for $x \in \Cc$} the set corresponds
to~$K_{\text{proj}}$.

When $m=1$ and $k=0$, an $(m, k)$-vector control barrier function is
equivalent to the usual notion of a control barrier
function~\cite[Definition~2]{ADA-SC-ME-GN-KS-PT:19}.  The
generalization provided by VCBFs allows us to consider a broader class
of safety sets. 

\begin{lemma}[VCBF-based Safeguarding
    Feedback]\label{lem:safe-feedback}
    Consider the system \eqref{eq:control-system} with safety set
    $\Cc$ and suppose $(g, h)$ is a vector control barrier function
    for $\Cc$ on $X$ relative to $\Uc$ {and MFCQ holds on $\Cc$}.
    If $\kappa:X \to \Uc$ is a feedback controller such that (i)
      $\kappa(x) \in K_{\textnormal{cbf}, \alpha}(x)$ for all
      $x \in X$ and (ii) for all initial conditions on $X$, the closed-loop system
      $\dot{x} = \Fc(x, \kappa(x))$ admits a unique absolutely continuous 
      solution satisfying the dynamics almost everywhere, 
      then $\kappa$ is safeguarding.
\end{lemma}

\new{The proof of the Lemma \ref{lem:safe-feedback} is omitted for
  space, but the result follows by Nagumo's Theorem
  \cite[Theorem~2]{JPA-AC:84} since one can show that the closed-loop
  dynamics are contained in the tangent cone.}  To synthesize a
safeguarding feedback controller, one can pursue a design using a
similar approach to Section~\ref{sec:proj-feedback}. Given a cost
function $J:X \times \Uc \to \real$, we let $\kappa(x)$ solve the
following mathematical program:
\begin{equation}
  \label{eq:qp-cbf-controller}
  \kappa(x) \in \underset{\nu \in K_{\textnormal{cbf}, \alpha}(x)}{\text{argmin}}
  J(x, \nu).  
\end{equation}
Similarly to the case of projection-based safeguarding feedback
control, care must be taken to verify the existence and uniqueness of
solutions to the closed-loop system, as well as to handle situations
where~\eqref{eq:qp-cbf-controller} does not have unique solutions. If
these properties hold, then the control design addresses the
challenges of projection-based methods. In particular, we can ensure
that a controller of the form \eqref{eq:qp-cbf-controller} is
well-defined outside the safety set and results in closed-loop system
with continuous solutions, under mild conditions which we discuss
in the following sections.
\vspace{-2.0ex}

\section{Problem Formulation}
Consider a variational inequality $\VI(F, \Cc)$ defined by a
continuously differentiable map $F:\real^n \to \real^n$ and a convex
set $\Cc$ of the form \eqref{eq:constraint-set}, where
$g:\real^n \to \real^m$ is continuously differentiable.  \new{We
  assume throughout that $\SOL(F, \Cc) \neq \emptyset$}.  Our goal is
to synthesize a solver for this problem in the form of a
continuous-time dynamical system.

\begin{problem}[Anytime solver of variational inequality]
  \label{problem:problem}
  Design a dynamical system, $\dot{x} = \Gc(x)$, which is well-defined
  on a set $X$ containing $\Cc$ such that
  \begin{enumerate}
  \item Trajectories of the system converge to $\SOL(F, \Cc)$;
  \item $\Cc$ is forward invariant;
  \item \new{$\Cc$ is asymptotically stable as a set.}
  \end{enumerate}
\end{problem}

Item (i) ensures that \emph{the dynamical system can be viewed as a 
solver of the problem~\eqref{eq:variational-inequality}}: solutions
can be obtained by simulating system trajectories and taking the limit
as~$t \to \infty$ of~$x(t)$. \new{Furthermore, the solver can be implemented 
on computational platforms by discretizing the dynamics}. 
Item (ii) ensures that this solver is
\emph{anytime}, meaning that even if terminated early, it is
guaranteed to return a feasible solution provided the initial
condition is feasible. Item (iii) accounts for infeasible 
initial conditions, and ensures asymptotic safety. 
Both the expression of the algorithm in the
form of a continuous-time dynamical system and the anytime property
are particularly useful for real-time applications, where the
algorithm might be interconnected with other physical processes --
e.g., when the algorithm output is used to regulate a physical plant
and constraints of the optimization problem ensure the safe operation
of the plant.

In the following, we introduce three dynamics to solve
Problem~\ref{problem:problem}. The first, synthesized using the
technique outlined in Section~\ref{sec:proj-feedback}, is the
\emph{projected monotone flow}. These dynamics are already well-known,
but we reinterpret them here through the lens of control theory. The
next two, synthesized using the technique outlined in
Section~\ref{sec:cbf-feedback}, are the \emph{safe monotone flow} and
the \emph{recursive safe monotone flow}. Both systems are entirely
novel. 


\section{Projected Monotone Flow}\label{sec:pmf}
Here, we discuss our first solution to Problem~\ref{problem:problem},
in the form of the \emph{projected monotone flow}.  We show that the
system can be viewed in two equivalent ways: either as a control
system with a feedback controller designed using the strategy outlined
in Section~\ref{sec:proj-feedback}, or as a projected dynamical
system.  In fact, this system admits many other equivalent
descriptions, for example in terms of monotone differential
inclusions, or complementarity
systems~\cite{BB-AD-CL-VA:06,WPMHH-JMS-SW:00,JPA-AC:84} \new{(see
  \cite{BB-AT:20} for a modern survey on the relationships between
  systems in various formalisms)}. Its properties have been
extensively studied~\cite{AN-DZ:96} \new{and techniques for
  computational implementation have been considered in
  \cite{VA-BB:08}}.
However, we focus here on the control-based and projection-based
forms.  In the coming sections we describe in detail the derivation of
each implementation, show their equivalence, and discuss the
properties of the resulting flow regarding safety and stability.

\subsection{Control-Based Implementation}\label{sec:pmf-control}

Our design strategy originates from the observation that, when $F$ is
monotone, the system $\dot{x} = -F(x)$ finds solutions to the
unconstrained variational inequality $\VI(F, \real^n)$. However,
trajectories flowing along this dynamics might leave the constraint
set~$\Cc$. This leads us to consider the control-affine
system:
\begin{equation}
  \begin{aligned}
    \label{eq:variational-system}
    \dot{x} &= \Fc(x, u, v) \\
    &= -F(x) - \sum_{i=1}^{m}u_i\nabla g_i(x) -
    \sum_{j=1}^{k}v_j\nabla h_j(x).
  \end{aligned}
\end{equation}
Here, we have augmented the system with inputs from the admissible set
$\Uc = \real^{m}_{\geq 0} \times \real^{k}$ to modify the flow of the
original drift $-F$ to account for the constraints in a way that
ensures that the solutions to \eqref{eq:variational-system} stay
inside of or approach~$\Cc$. The idea is that if the constraint
$g_i(x) \leq 0$ is in danger of being violated, the corresponding
input $u_i$ can be increased to ensure trajectories continue to
satisfy it.  Likewise, the input $v_j$ can be increased or decreased
to ensure the corresponding constraint $h_j(x) = 0$ is satisfied along
trajectories.

Our design proceeds by thinking of $\Cc$ as a safety set for the
system and using the approach outlined in
Section~\ref{sec:proj-feedback} to synthesize a safeguarding feedback
controller $(u, v) = \kappa(x)$.  Assuming that MFCQ holds for all $x \in \Cc$,
the map $K_\text{proj}:\real^n \rightrightarrows \real^{m}_{\geq 0} \times
\real^{k}$ takes the form
\begin{equation}
  \label{eq:Kproj}
  \begin{aligned}
    K_\text{proj}(x)
    =\Big\{ (u, v) \in \real^m_{\geq 0} \times \real^k \;\Big|\;
    &-\pfrac{g_{I_0}}{x}F(x)-\pfrac{g_{I_0}}{x}\pfrac{g}{x}^\top u
      -\pfrac{g_{I_0}}{x}\pfrac{h}{x}^\top v \leq 0,
    \\ 
    &-\pfrac{h}{x}F(x)-\pfrac{h}{x}\pfrac{g}{x}^\top u
      -\pfrac{h}{x}\pfrac{h}{x}^\top v = 0 
      \Big\}.  
  \end{aligned}
\end{equation}

\new{The set $K_\text{proj}(x)$ consists of the inputs such that the 
dynamics \eqref{eq:variational-system} lie in $T_\Cc(x)$}. 
The following result states that the set of admissible controls is
nonempty.  We omit its proof for space reasons, but note that it
readily follows from Farka's Lemma~\cite{RTR:70}.

\begin{lemma}\longthmtitle{Projection onto Tangent
    Cone is Feasible}
  If $x \in \Cc$ and MFCQ holds at $x$, then
  $K_\text{proj}(x) \neq \emptyset$.
\end{lemma}

We then use the feedback controller
\begin{equation}\label{eq:qp-proj-controller-u-v}
  \kappa(x) \in \underset{(u, v) \in K_\text{proj}(x)}{\text{argmin}} 
  J(x, u, v) ,
\end{equation}
where we set the objective function to be 
\begin{equation}
  \label{eq:objective}
  J(x, u, v) = \frac{1}{2}\bigg\|  \sum_{i=1}^{m}u_i\nabla g_i(x) +
  \sum_{j=1}^{k}v_j\nabla h_j(x)  \bigg\|^2. 
\end{equation}
This function measures the magnitude of the ``modification'' of the
drift term in~\eqref{eq:variational-system}.  Thus, 
the quadratic program (QP)~\eqref{eq:qp-proj-controller-u-v} has the interpretation,
at each $x$, of finding the control input such that the closed-loop
system dynamics are as close as possible to~$-F(x)$, while still being
in $T_\Cc(x)$. In general, the program given
by~\eqref{eq:qp-proj-controller-u-v} does not have unique solutions.
Despite this, we show below that the closed-loop dynamics of
\eqref{eq:variational-system} is well-defined regardless of which
solution to \eqref{eq:qp-proj-controller} is chosen. We refer to it as
the \emph{projected monotone flow}. 

\subsection{Projection-Based Implementation}
The second implementation of the projected monotone flow consists of
projecting $-F(x)$ onto the tangent cone of the constraint set.  In
general, the tangent cone does not have an \new{explicit
  parameterization} that allows us to compute the projection
easily. However, when the appropriate constraint qualification
condition holds, \new{the tangent cone can be parameterized by
  affine constraints} which allow the projection to be
implemented as a quadratic program. Let $x \in \Cc$ and suppose that
MFCQ holds at $x$. It follows that the tangent cone can be
parameterized\footnote{\new{
  The right-hand side of~\eqref{eq:parameterization-tc} is called 
  the \emph{linearization cone of~$\Cc$}, and denoted as $T^{\text{lin}}_{\Cc}(x)$. 
  In general, $T^{\text{lin}}_{\Cc}(x) \subset
  T_\Cc(x)$, and we say that \emph{Abadie's Constraint Qualification}
  (ACQ) \cite{DWP:73} holds at $x \in \Cc$ if the reverse inclusion
  holds, i.e., $T^{\text{lin}}_{\Cc}(x) = T_\Cc(x)$.  MFCQ is a stronger
  condition that implies ACQ. Many of the results we state in the
  paper can be generalized to hold under the weaker assumption of ACQ,
  however for simplicity, we avoid a detailed technical discussion of
  constraint qualifications.
}}
as
\begin{align}\label{eq:parameterization-tc}
  T_\Cc(x) = \bigg\{ \xi \in \real^n \;\Big|\; \pfrac{h(x)}{x}\xi = 0,
  \pfrac{g_{I_0}(x)}{x}\xi \leq 0 \bigg\}.  
\end{align}
The projection-based implementation of the projected monotone flow
takes then the following form:
\begin{equation}
  \label{eq:proj-grad}
  \begin{aligned}
    \dot{x} &= \proj{T_\Cc(x)}{-F(x)}
    \\
    &= \underset{\xi \in \real^n}{\text{argmin}} \quad 
    \frac{1}{2}\norm{\xi + F(x)}^2
    \\
    & \quad \; \text{subject to} \quad \pfrac{g_{I_0}(x)}{x}\xi \leq 0,
    \pfrac{h(x)}{x}\xi = 0.
  \end{aligned}
\end{equation}

The projection onto the tangent cone ensures that~$\Cc$ is forward invariant, 
by Nagumo's Theorem~\cite[Theorem~3.1]{FB:99}.

\subsection{Properties of Projected Monotone Flow}

Here, we lay out the properties of the projected monotone flow.  We
begin by establishing the equivalence between the control- and
projection-based implementations. We then discuss existence and
uniqueness of solutions, and finally the stability and safety
properties of the dynamics.

\subsubsection{Equivalence of Control-Based and Projection-Based
  Implementations}
Equivalence follows directly from the properties of the tangent cone,
as we show next.

\begin{proposition}\longthmtitle{Equivalence of Control-Based and
    Projected-Based Implementations}\label{prop:proj-equivalence}
  \new{Assume MFCQ holds at $x \in \Cc$ and let $(u, v)$ be any solution
  to~\eqref{eq:qp-proj-controller-u-v}. 
  Then, $\Fc(x, u, v) = \proj{T_\Cc(x)}{-F(x)}$.}
\end{proposition}
\begin{proof}
  Let $(u, v)$ be any solution to~\eqref{eq:qp-proj-controller-u-v} and
  $\xi = \proj{T_\Cc(x)}{-F(x)}$. Then $\Fc(x, u, v) \in T_{\Cc}(x)$, 
  so it follows immediately by optimality of $\xi$ that
  \[
    \norm{\xi + F(x)}^2 \leq \norm{\Fc(x, u, v) + F(x)}^2.
  \]
  Next, because $\xi$ is given by a projection, there exists
  $w \in N_\Cc(x)$ such that $\xi + F(x) + w = 0$, see
  e.g.,~\cite[Corollary 2]{BB-AD-CL-VA:06}.  If MFCQ holds at
  $x \in \Cc$, by \cite[Theorem~6.14]{RTR-RJBW:98}, there exists
  $(\bar{u}, \bar{v})$ such that 
  \[ 
    \begin{aligned}
      &w = \sum_{i=1}^{m}\bar{u}_i\nabla g_i(x) +
      \sum_{j=1}^{k}\bar{v}_j\nabla h_j(x), \qquad \bar{u} \geq 0 , \quad
      \bar{u}^\top g(x) = 0.
    \end{aligned}
  \]
  Combining this expression with the fact that
  $\xi = -F(x) - w \in T_\Cc(x)$ and using the parameterization of the
  tangent cone in~\eqref{eq:parameterization-tc}, we deduce that
  $(\bar{u}, \bar{v}) \in K_\text{proj}(x)$.  By optimality of
  $(u, v)$,
  \begin{align*}
    \norm{\xi + F(x)}^2 &= \bigg\|  \sum_{i=1}^{m}\bar{u}_i\nabla g_i(x) +
      \sum_{j=1}^{k}\bar{v}_j\nabla h_j(x) \bigg\|^2 \\
      &\geq \bigg\|  \sum_{i=1}^{m}u_i\nabla g_i(x) +
      \sum_{j=1}^{k}v_j\nabla h_j(x)  \bigg\|^2 \\
      &= \norm{\Fc(x, u, v) + F(x)}^2.
  \end{align*}
  Thus, $\Fc(x, u, v)$ is a solution to \eqref{eq:proj-grad}. 
  But since, by convexity, the projection onto the
  tangent cone must be unique, $\xi = \Fc(x, u, v)$.  
\end{proof}

The value of Proposition~\ref{prop:proj-equivalence} stems from
showing that safety-critical control can be used to 
design algorithms that solve variational inequalities.  Though the
control strategy pursued in Section~\ref{sec:pmf-control} results in a
known flow, this sets up the basis for employing other design
strategies from safety-critical control to yield novel methods, as we
show later.

\subsubsection{Existence and Uniqueness of Solutions}
The projected monotone flow is discontinuous, and hence one must
consider notions of solutions beyond the classical ones, see
e.g.,~\cite{JC:08-csm-yo}. Here, we consider Carath\'eodory solutions,
which are absolutely continuous functions that
satisfy~\eqref{eq:proj-grad} almost everywhere. The existence and
uniqueness of solutions for all initial conditions follows readily
from~\cite[Chapter 3.2, Theorem~1(i)]{JPA-AC:84} or
\new{\cite[Theorem~1]{BB-AD-CL-VA:06}}.

\subsubsection{Safety and Stability of Projected Monotone Flow}
We now show that the projected monotone flow is safe, meaning that the
constraint set $\Cc$ is forward invariant, and the solution set
$\SOL(F, \Cc)$ is stable.  Forward invariance of $\Cc$ follows
directly from Lemma~\ref{lem:proj-feedback}. The equilibria of the projected
monotone flow correspond to solutions to~$\VI(F, \Cc)$. Finally,
stability of a solution $x^*$ can be certified using the Lyapunov
function 
$ V(x) = \frac{1}{2}\norm{x - x^*}^2$,
as a consequence
of~\cite[Chapter 3.2, Theorem 1(ii)]{JPA-AC:84}
\new{or~\cite[Theorems~1 and~2]{DG-BB:04}}.  These properties are
summarized next.

\begin{theorem}\longthmtitle{Safety and Stability Properties of 
  Projected Monotone Flow}\label{thm:proj-stability}
  Let $\Cc$ be convex and suppose MFCQ holds for all $x \in \Cc$. The
  following hold for the projected monotone flow:
  \begin{enumerate}
  \item $\Cc$ is forward invariant;
  \item $x^*$ is an equilibrium of the projected monotone flow if and
    only if $x^* \in \SOL(F, \Cc)$;
  \item If $x^* \in \SOL(F, \Cc)$ and $F$ is monotone, then $x^*$ is
    globally Lyapunov stable relative to $\Cc$;
  \item If $F$ is $\mu$-strongly monotone, then the projected 
    monotone flow is contracting at rate $\mu$. In particular, 
    the unique solution $x^* \in \SOL(F, \Cc)$ is 
    globally exponentially stable relative to $\Cc$.
  \end{enumerate}
\end{theorem}

\section{Safe Monotone Flow}\label{sec:smf}
In this section, we discuss a second solution to
Problem~\ref{problem:problem}, which results in an entirely novel
flow, termed \emph{safe monotone flow}.  Similar to the projected
monotone flow, this system admits two equivalent implementations:
either as a control-system with a safeguarding feedback controller or
as a projected dynamical system.  \new{Because of its continuity
  properties, the system can be solved using standard ODE
  discretization schemes, cf. \cite[Remark 5.6]{AA-JC:24-tac}, but we
  leave the formal analysis of this for future work.}

\subsection{Control-Based Implementation}
We start with the control system~\eqref{eq:variational-system} with
the admissible control set $\Uc = \real^{m}_{\geq 0} \times \real^k$,
viewing $\Cc$ as a safety set, and design a safeguarding
controller. We synthesize this controller using the function $(g, h)$
as a VCBF, following the approach outlined in
Section~\ref{sec:cbf-feedback}.

Letting $\alpha > 0$ be a parameter, the set of control inputs
ensuring safety is given by
\[
  \begin{aligned}
    K_{\textnormal{cbf}, \alpha}(x) = \Big\{(u, v) \in \real_{\geq 0}^m
      \times \real^k \;\Big|\; &-\pfrac{g}{x}F(x) -\pfrac{g}{x}\pfrac{g}{x}^\top u
      -\pfrac{g}{x}\pfrac{h}{x}^\top 
      v \leq  -\alpha g(x)
    \\
    & -\pfrac{h}{x}F(x) -\pfrac{h}{x}\pfrac{g}{x}^\top u -
      \pfrac{h}{x}\pfrac{h}{x}^\top 
      v = -\alpha h(x) \Big\}.
  \end{aligned}
\]
The next result shows that this set is nonempty on an open set
containing~$\Cc$.

\begin{lemma}[Vector Control Barrier Function
    for \eqref{eq:variational-system}]
  \label{lem:vcbf}
  Assume MFCQ holds for all $x \in \Cc$. Then there exists an open set
  $X \supset \Cc$ on which $\phi=(g, h)$ is a vector-control
  barrier function of \eqref{eq:variational-system} for $\Cc$, on $X$,
  relative to $\real^{m}_{\geq 0} \times \real^k$.
\end{lemma}

The proof of this result is identical to \cite[Lemma
4.1]{AA-JC:24-tac} and we omit it for brevity. By
Lemma~\ref{lem:vcbf}, the feedback controller $(u, v) = \kappa(x)$ where
\begin{equation}\label{eq:smf-controller}
  \kappa(x) \in \underset{(u, v) \in K_{\textnormal{cbf}, \alpha}(x)}{\text{argmin}} 
  J(x, u, v),
\end{equation}
and $J$ is given by~\eqref{eq:objective}, is well-defined on $X$. This
controller has the same interpretation as before: determining the
control input belonging to~$K_{\textnormal{cbf}, \alpha}(x)$ such that the
closed-loop system dynamics are as close as possible to~$-F(x)$.
Similar to the case with projection-based methods, the problem
\eqref{eq:qp-cbf-controller} does not necessarily have unique
solutions. However, we show below that the closed-loop system is
well-defined regardless of which solution is chosen. We refer to it as
the \emph{safe monotone flow with safety parameter $\alpha$}, 
denoted~$\dot{x} = \Gc_\alpha(x)$.

\subsection{Projection-Based Implementation}
Here we describe the implementation of the safe monotone flow as a
projected dynamical system. Similar to the projected monotone flow,
the projected system is obtained by projecting $-F(x)$ onto a
set-valued map.  However, because the projection onto the tangent cone
is in general discontinuous as a function of the state, we replace the
tangent cone with the \emph{$\alpha$-restricted tangent set}
\new{(here $\alpha > 0$)}, denoted $T_\Cc^{(\alpha)}$, defined as
\begin{equation}
  \label{eq:approximate-tangent-cone}
  \begin{aligned}
    T^{(\alpha)}_{\Cc}(x) = \Big\{ \xi \in \real^n \;\Big|\;\pfrac{g(x)}{x}\xi \leq -\alpha g(x), \pfrac{h(x)}{x}\xi =
    -\alpha h(x) \Big\} .
  \end{aligned}
\end{equation}
Figure~\ref{fig:approx-tangent} illustrates this definition.  This set
can be interpreted as an approximation of the usual tangent cone, but
differs in several key ways.  First, the restricted tangent set is not
a cone, meaning that vectors in~$T^{(\alpha)}_{\Cc}(x)$ cannot be
scaled arbitrarily: in certain directions, the magnitude of vectors in
$T^{(\alpha)}_{\Cc}(x)$ is restricted. 
\new{Second, unlike the tangent cone,
$T_\Cc^{(\alpha)}:\real^n \rightrightarrows \real^n$ is 
a continuous map.}
\new{Finally, while the tangent cone is empty for 
all~$x \not\in \Cc$}, this is not the case for the restricted tangent
set. In fact, it can be shown that $T_\Cc^{(\alpha)}$ takes nonempty
values on an open set containing~$\Cc$. This property allows for the
safe monotone flow to be well-defined for infeasible initial
conditions.  The next result states properties of the
$\alpha$-restricted tangent set.

\pagebreak

\begin{proposition}\longthmtitle{Properties of $\alpha$-Restricted
    Tangent Set}
  \label{prop:restricted-tangent-properties}
  Assume MFCQ holds for all $x \in \Cc$. The set-valued map
  $T_\Cc^{(\alpha)}:\real^n \rightrightarrows \real^n$ satisfies:
  \begin{enumerate}
  \item \new{$T_\Cc^{(\alpha)}$ is a continuous set-valued map;}
  \item $T_\Cc^{(\alpha)}(x)$ is convex for all $x \in \real^n$;
  \item For all $x \in \Cc$, $T^{(\alpha)}_\Cc(x)$ satisfies MFCQ at
    all $\xi \in T^{(\alpha)}_\Cc(x)$.
  \item There exists an open set $X$ containing $\Cc$ such that
    $T^{(\alpha)}_\Cc(x) \neq \emptyset$ for all $x \in X$;
  \item If $x \in \Cc$, then \new{$T_\Cc^{(\alpha)}(x) \subseteq
    T_\Cc(x)$}.
  \end{enumerate}
\end{proposition}
\begin{proof}
  %
  \new{We observe that (i) follows since
    $T_\Cc^{(\alpha)}:\real^n \rightrightarrows \real^n$ is outer
    semicontinuous by \cite[Example 5.8]{RTR-RJBW:98}, and inner
    semicontinuous by \cite[Theorem 5.9(a)]{RTR-RJBW:98}.  Next, (ii)}
  follows from the fact that the constraints characterizing
  $T_C^{(\alpha)}(x)$ are affine in the variable $\xi$. We prove
  \new{(iii)} using the same strategy as \cite[Lemma
  4.5]{AA-JC:24-tac}, which we sketch here. If MFCQ holds at
  $x \in \Cc$, then the inequalities
  defining~\eqref{eq:approximate-tangent-cone} satisfy Slater's
  condition \cite[Chapter~5.2.3]{SB-LV:09} at $x$ and therefore MFCQ
  holds for all $\xi \in T^{(\alpha)}_\Cc(x)$.  To show \new{(iv)}, we
  note that Slater's condition implies that the affine constraints
  parameterizing $T^{(\alpha)}_\Cc(x)$ are \emph{regular}
  \cite[Theorem~2]{SMR:75}, meaning that the system remains feasible
  with respect to perturbations.  Since
  $T^{(\alpha)}_\Cc(x) \neq \emptyset$ for all $x \in \Cc$, there
  exists an open set $X$ containing $\Cc$ such that
  $T^{(\alpha)}_\Cc(x) \neq \emptyset$ for all $x\in X$.  Finally,
  \new{(v)} follows from the definition of the tangent cone. 
\end{proof}

\begin{figure}[!h]
  \centering
  \subfigure[]{
    \includegraphics[width=0.39\linewidth]{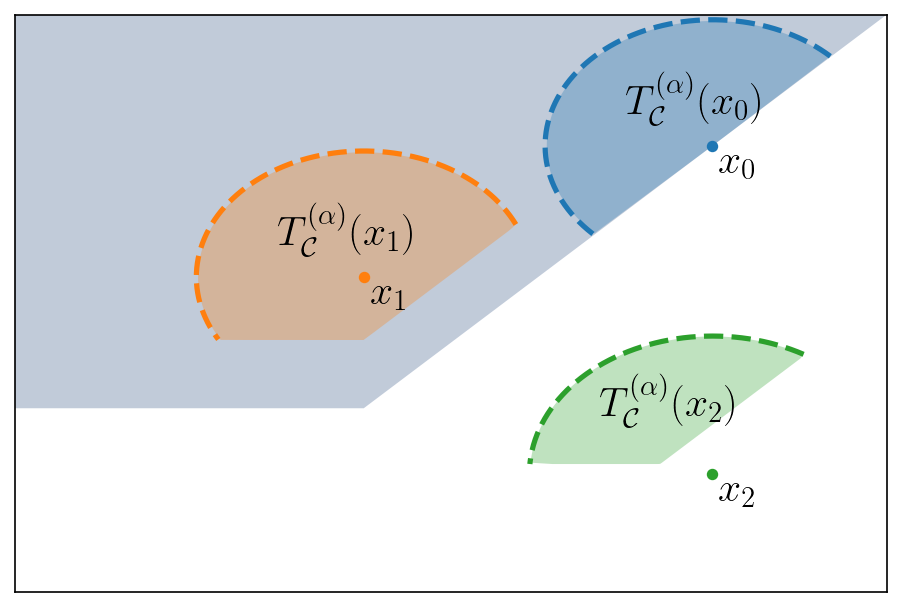}}
  \hspace{6.0ex}
  \subfigure[]{
    \includegraphics[width=0.39\linewidth]{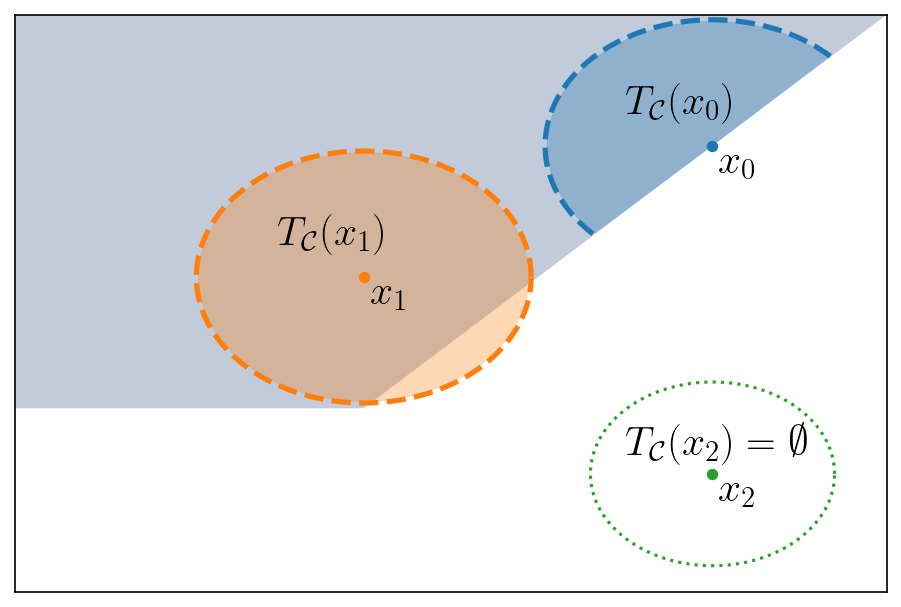}}
  \caption{Illustration of the notion of (a) $\alpha$-restricted
  tangent set and (b) tangent cone. The gray-shaded region
  represents the set~$\Cc$.  The colored regions depict either type
  of set, which consists of vectors centered various points
  $x_i$. The dashed border indicates directions in which the
  magnitude of vectors in the set are unbounded.  In (a), note that
  the $\alpha$-restricted tangent set is nonempty at
  $x_2 \not\in \Cc$, however because the region does not overlap
  with the point $x_2$, the set $T^{(\alpha)}_\Cc(x_2)$ does not
  contain the zero vector, and all vectors point strictly
  toward~$\Cc$.  In (b), note that the tangent cone is \new{empty}
  at points outside $\Cc$. }
  \label{fig:approx-tangent}
\end{figure}

Using the $\alpha$-restricted tangent set, we can define the projected
dynamical system
\begin{equation}
  \label{eq:smooth-proj-grad}
  \begin{aligned} 
    \dot{x}
    = \proj{T^{(\alpha)}_\Cc(x)}{-F(x)}
    & = \underset{\xi \in \real^n}{\text{argmin}} \quad 
      \frac{1}{2}\norm{\xi + F(x)}^2
    \\
    &
      \quad \text{subject to} \quad \pfrac{g(x)}{x}\xi \leq -\alpha g(x)
    \\
    & \quad \phantom{subject to} \quad \pfrac{h(x)}{x}\xi = -\alpha h(x) .
  \end{aligned}
\end{equation}
Similar to the projected monotone flow, the projection operation
ensures that the trajectories of the system remain in the safety
set. However, as we show next, the advantages of projecting onto the
restricted tangent cone is that the system is well-defined for
infeasible initial conditions, and trajectories of the system are
smooth.

\subsection{Properties of Safe Monotone Flow}
We now discuss the properties of the safe monotone flow. We begin by
establishing the equivalence of the control-based and projection-based
implementations. Next, we discuss its stability and safety properties.

\subsubsection{Equivalence of Control-Based and Projection-Based
  Implementations}
We establish here that the control-based and projection-based
implementations of the safe monotone flow are equivalent. The next
result states that the closed-loop dynamics resulting from the
implementation of the controller~\eqref{eq:qp-cbf-controller}
over the system~\eqref{eq:variational-system} is equivalent to the projection
onto $T^{(\alpha)}_\Cc(x)$. The structure of the proof 
mirrors that of Proposition \ref{prop:proj-equivalence}. 

\begin{proposition}\longthmtitle{Equivalence of Control-Based and
    Projection-Based
    Implementations}\label{prop:smf-equivalent-formulation}
  Assume MFCQ holds for all $x \in \Cc$ and let
  $X \subset \real^n$ be an open set containing $\Cc$ on which
  $K_{\textnormal{cbf}, \alpha}$ takes nonempty values. Let $(u, v)$ be any
  solution to \eqref{eq:smf-controller} at $x \in X$ (note that
  $\Gc_\alpha(x) = \Fc(x, u, v)$).  Then, $\Gc_\alpha(x) = \proj{T^{(\alpha)}_\Cc(x)}{-F(x)}$. 
\end{proposition}
\begin{proof}
  Let $(u, v)$ be any solution to~\eqref{eq:smf-controller} and
  $\xi = \proj{T^{(\alpha)}_\Cc(x)}{-F(x)}$. Then
  $\Fc(x, u, v) \in T^{(\alpha)}_{\Cc}(x)$, so it follows immediately
  by optimality of $\xi$ that
  \[ \norm{\xi + F(x)}^2 \leq \norm{\Fc(x, u, v) + F(x)}^2.\]  
  Next,
  since $\xi$ is given by a projection, there exists
  $w \in N_{T}(\xi)$, where $T = T^{(\alpha)}_\Cc(x)$, such that
  $\xi + F(x) + w = 0$, see e.g.,~\cite[Corollary 2]{BB-AD-CL-VA:06},
  and where
  \[ 
  \begin{aligned}
    &w = \sum_{i=1}^{m}\bar{u}_i\nabla g_i(x) +
      \sum_{j=1}^{k}\bar{v}_j\nabla h_j(x), \quad \bar{u} \geq 0, \qquad \bar{u}^\top (\nabla g(x)^\top + \alpha g(x)) = 0.
  \end{aligned}
  \] 
  Combining this expression with the fact that
  $\xi = -F(x) - w \in T^{(\alpha)}_\Cc(x)$ and using the definition
  of the $\alpha$-restricted tangent cone, we deduce that
  $(\bar{u}, \bar{v}) \in K_{\textnormal{cbf}, \alpha}(x)$.  By optimality
  of $(u, v)$, we have
  \begin{align*}
    \norm{\xi + F(x)}^2 &= \bigg\|  \sum_{i=1}^{m}\bar{u}_i\nabla g_i(x) +
      \sum_{j=1}^{k}\bar{v}_j\nabla h_j(x) \bigg\|^2 \\
    &\geq \bigg\| \sum_{i=1}^{m}u_i\nabla g_i(x) +
      \sum_{j=1}^{k}v_j\nabla h_j(x) \bigg\|^2 \\
      &= \norm{\Fc(x, u, v) + F(x)}^2.
  \end{align*}
  Thus, $\Fc(x, u, v)$ is a solution to \eqref{eq:smooth-proj-grad}. But since, by convexity, the projection onto the $\alpha$-restricted tangent set
  is unique, $\xi = \Fc(x, u, v)$.  
\end{proof}

\subsubsection{Existence and Uniqueness of Solutions}
We now discuss conditions for the existence and uniqueness of
solutions of the safe monotone flow.

\begin{proposition}\longthmtitle{Existence and Uniqueness of Solutions to Safe Monotone Flow}
  \label{prop:existence-uniqueness}
  Assume MFCQ and the constant-rank condition hold on $\Cc$ for all
  $x \in \Cc$ and let $X$ be the open set containing $\Cc$ in
  Proposition~\ref{prop:restricted-tangent-properties}(iii).  Then
  \begin{enumerate}
  \item For all $x_0 \in \Cc$, there exists a unique solution
    $x:[0, \infty) \to \real^n$ to the safe monotone flow with
    $x(0) = x_0$.
    \item For all $x_0 \in X$, there exists a unique solution $x:[0, t_f] \to \real^n$ 
    such that $x(0) = x_0$. Furthermore, the solution can be extended so that either $t_f = \infty$ 
    or $x(t) \to \partial X$ as $t \to t_f$. 
  \end{enumerate}
\end{proposition}

\begin{proof}
  We first note that the program \eqref{eq:smooth-proj-grad}
  satisfies the General Strong Second-Order Sufficient Condition
  (cf. \cite[pp. 1047]{JL:95}) and Slater's condition at $x \in X$. Because the
  objective function and constraints of \eqref{eq:smooth-proj-grad}
  are twice continuously differentiable, we can apply \cite[Theorem
  3.6]{JL:95} to conclude that $\Gc_\alpha$ is locally Lipschitz at
  $x$. Therefore, $\Gc_\alpha$ is also lower semicontinuous and
  by \cite[Chapter 2, Theorem~1]{JPA-AC:84} there exists for all
  $x_0 \in X$ a solution $x:[0, t_f] \to \real^n$ for some $t_f > 0$ with
  $x(0) = x_0$. Furthermore, either $t_f = \infty$ or
  $x(t) \to \partial X$ as $t \to t_f$. Uniqueness of solutions 
  holds by local Lipschitzness and (ii) follows. 

  To show (i), we note that $\Gc_\alpha(x) \in T_\Cc(x)$, and by
  \cite[Theorem~3.1]{FB:99}, for any solution with $x(0) \in \Cc$, we
  have that $x(t) \in \Cc$ for all $t \geq 0$ on the interval on which
  the solution exists. Since $\Cc \subset \text{int}(X)$, solutions
  beginning in $\Cc$ cannot approach $\partial X$, and exist for all
  time. 
\end{proof}

\subsubsection{Safety  of Safe Monotone Flow}
Here we establish the safety properties of the safe monotone flow. We
begin by characterizing optimality conditions for the closed-loop
dynamics.

\begin{lemma}\longthmtitle{Optimality Conditions for Closed-loop
    Dynamics}\label{lem:optimality-conditions}
  For $x \in \real^n$, consider the equations
  \begin{subequations}
    \label{eq:KKT_proj}
    \begin{align}
      \label{eq:dlagrangian}
      \xi + F(x) + \pfrac{g(x)}{x}^\top u + \pfrac{h(x)}{x}^\top v &=0, 
      \\
    \label{eq:pinequality}
      \pfrac{g(x)}{x}\xi + \alpha g(x) &\leq 0 , \\
      \label{eq:pequality}
      \pfrac{h(x)}{x}\xi + \alpha h(x) &= 0,
      \\
      \label{eq:nonnegativity}
      u &\geq 0, \\
      \label{eq:transversality}
      u^\top \bigg(\pfrac{g(x)}{x}\xi + \alpha g(x)\bigg) &= 0 ,
    \end{align}
  \end{subequations}
  in $(\xi, u, v)$. Let
  $\Lambda_\alpha:\real^n \rightrightarrows \real^m_{\geq 0} \times
  \real^k$ be
  \[
    \Lambda_\alpha(x) = \{ (u, v) \mid \exists \xi \text{ such that }
    (\xi, u, v) \text{ solves~\eqref{eq:KKT_proj}} \}.
  \]
  Assume MFCQ holds for all $x \in \Cc$. Then, there exists an open
  set $X \supset \Cc$ such that, if $x \in X$, then
  $\Lambda_\alpha(x) \neq \emptyset$.  If $(\xi, u, v)$ solves
  \eqref{eq:KKT_proj}, then $\Gc_\alpha(x) = \xi$ and $(u, v)$
  solves~\eqref{eq:smf-controller}.
\end{lemma}
\begin{proof}
  Let $\tilde{F}(x, \xi) = F(x) + \xi$. Then
  $\xi = \proj{T^{(\alpha)}_\Cc(x)}{-F(x)}$
  is a solution to the
  monotone variational inequality
  $\VI(\tilde{F}(x, \cdot), T^{(\alpha)}_\Cc(x))$, parameterized by
  $x$. Since MFCQ holds at all $\xi \in T^{(\alpha)}_\Cc(x)$ by
  Proposition \ref{prop:restricted-tangent-properties}(iii), we can
  use the KKT conditions to characterize $\Gc_\alpha(x)$, which
  correspond to \eqref{eq:KKT_proj}. Further, by Proposition
  \ref{prop:restricted-tangent-properties}(iv), solutions to
  \eqref{eq:KKT_proj} exist on an open set $X$ containing $\Cc$. Since
  $\tilde{F}$ is strongly monotone with respect to $\xi$, the solution
  to $\VI(\tilde{F}(x, \cdot), T^{(\alpha)}_\Cc(x))$ is unique,
  proving the result.  
\end{proof}

We rely on the optimality conditions in
Lemma~\ref{lem:optimality-conditions} to establish the following
result characterizing the equilibria and safety properties of the 
safe monotone flow.

\begin{theorem}\longthmtitle{Equilibria and Safety of Safe Monotone
    Flow}\label{thm:safe-stability}
  Let $\alpha > 0$, $\Cc$ be convex, and suppose MFCQ and the constant
  rank condition hold for all $x \in \Cc$.  The following hold for
  the safe monotone flow:
  \begin{enumerate}
  \item $\Cc$ is forward invariant and asymptotically stable on $X$;
    \label{item:safety}
  \item $x^*$ is an equilibrium if and only if $x^* \in \SOL(F, \Cc)$;
    \label{item:eq}
  \end{enumerate}
\end{theorem}
\begin{proof}
  To show (i), note that by
  Proposition~\ref{prop:smf-equivalent-formulation}, for all $x \in X$
  there exists $(u(x), v(x)) \in K_{\textnormal{cbf}, \alpha}(x)$ such that
  $\Gc_\alpha(x) = \Fc(x, u(x), v(x))$.  Given the existence and
  uniqueness of solutions of the closed-loop system,
  cf. Propositions~\ref{prop:existence-uniqueness}, the
  result follows from Lemma~\ref{lem:safe-feedback} since
  $\phi(x) = (g(x), h(x))$ is a VCBF.
  Statement~(ii) follows from the observation that, if
  $\Gc_\alpha(x^*) = 0$, by Lemma~\ref{lem:optimality-conditions},
  there exists $(u^*, v^*)$ such that $(0, u^*, v^*)$ solves
  \eqref{eq:KKT_proj}, which holds if and only if $(x^*, u^*, v^*)$
  solves~\eqref{eq:KKT}.  
\end{proof}

\subsubsection{Stability of Safe Monotone Flow}

Here we characterize the stability properties of the safe monotone
flow. We begin by establishing conditions for stability relative
to~$\Cc$.

\begin{theorem}\longthmtitle{Stability of Safe Monotone Flow Relative
    to $\Cc$}
  \label{theorem:C-stability}
  Assume MFCQ \new{and the constant rank condition} hold for all $x \in \Cc$. Then
  \begin{enumerate}
    \item If $x^* \in \SOL(F, \Cc)$ and $F$ is monotone, then $x^*$ is
    globally Lyapunov stable relative to $\Cc$;
    \label{item:lyap}
    \item If $x^* \in \SOL(F, \Cc)$ and $F$ is $\mu$-strongly monotone,
    then $x^*$ is globally asymptotically stable relative to~$\Cc$.
    \label{item:as}
  \end{enumerate}
\end{theorem}

Before proving Theorem~\ref{theorem:C-stability}, we provide several
intermediate results. Our strategy relies on fixing
$x^* \in \SOL(F, \Cc)$ and considering the candidate Lyapunov function
\begin{align}\label{eq:C-lyap-candidate}
  V(x) \!=\!
  \underbrace{\frac{1}{2}\norm{x - x^*}^2}_{\tilde{V}(x)} \!-
  \frac{1}{\alpha^2}
  \!   \underbrace{\inf_{\xi \in T^{(\alpha)}_\Cc(x)} \! \{ \xi ^\top
  F(x) \! + \!
  \frac{1}{2}\norm{\xi}^2 \!\}}_{W(x)} \!.
\end{align}
We first present bounds on the Dini derivative of $\tilde{V}$
along~$\Gc_\alpha$. The result is proven in Appendix \ref{ap:V1}.

\begin{lemma}[Dini Derivative of $\tilde{V}$]
  \label{lem:V1}
  Assume MFCQ holds for all $x \in \Cc$. For $x^* \in \SOL(F, \Cc)$,
  let $(u^*, v^*)$ be Lagrange multipliers corresponding to $x^*$.
  For $x \in X$ and $(u, v) \in \Lambda_\alpha(x)$, then
  \begin{align*}
    \Dini_{\Gc_\alpha}\tilde{V}(x) \leq -\mu \norm{x - x^*}^2-(u - u^*)^\top (g(x) - g(x^*)) - (v - v^*)^\top h(x),
  \end{align*}
  if $F$ is $\mu$-strongly monotone (the inequality holds with $\mu=0$ if
  F is monotone instead).
\end{lemma}

We now move on to characterizing properties of~$W$. The following result 
is proven in Appendix~\ref{ap:W-properties}.

\begin{lemma}[Properties of $W$]
  \label{lem:W-properties}
  Assume MFCQ holds everywhere on~$\Cc$.  Define the
  matrix-valued function
  $Q:X \times \real^{m}_{\geq 0} \to \real^{n \times n}$ by
  \[ 
    Q(x, u) = \frac{1}{2}\bigg( \pfrac{F(x)}{x} + \pfrac{F(x)}{x}^\top
    \bigg) + \sum_{i=1}^{m}u_i\nabla^2 g_i(x).
  \] 
  Then, for all $x \in X$ and $(u, v) \in \Lambda_\alpha(x)$,
  \begin{equation}
    \label{eq:W-closed-form}
    W(x) = -\frac{1}{2}\norm{\Gc_\alpha(x)}^2 + \alpha u^\top g(x) +
    \alpha v^\top h(x) 
  \end{equation}
  and 
  \begin{equation}
    \label{eq:DW}
    \begin{aligned}
      \Dini_{\Gc_\alpha}W(x)
      & \geq \norm{\Gc_\alpha(x)}^2_{Q(x, u)}-\alpha^2 u^\top g(x) - \alpha^2 v^\top h(x). 
    \end{aligned}
  \end{equation}
\end{lemma}

We are now ready to prove Theorem~\ref{theorem:C-stability}.

\begin{proof}[Proof of Theorem~\ref{theorem:C-stability}]
  Let $x^* \in \SOL(F, \Cc)$ and suppose the hypotheses of (i)
  hold. Consider the function $V:X \to \real$ defined
  by~\eqref{eq:C-lyap-candidate}. We show that $V$ is a (strict)
  Lyapunov function when $F$ is ($\mu$-strongly) monotone.  Let
  $x \in \Cc$ and $(u, v) \in \Lambda_\alpha(x)$. 
  Then,
  \new{since $u \geq 0$, $g(x) \leq 0$, and $h(x) = 0$, it follows that
  $\alpha u^\top g(x) + \alpha v^\top h(x) \leq 0$}.
  Thus, by examining the
  expression in \eqref{eq:W-closed-form} we see that $W(x) \leq 0$ for
  all $x \in \Cc$ with equality if and only if $x \in \SOL(F, \Cc)$.
  Thus, $V$ is positive definite with respect to $x^*$. Next,
  \[\Dini_{\Gc_\alpha}V(x) = \Dini_{\Gc_\alpha}\tilde{V}(x) - \frac{1}{\alpha^2} \Dini_{\Gc_\alpha}W(x),\]
  and by Lemmas~\ref{lem:V1} and~\ref{lem:W-properties},
  \[
    \begin{aligned}
      \Dini_{\Gc_\alpha}V(x) \leq -\frac{1}{\alpha^2}
       \norm{\Gc_\alpha(x)}^2_{Q(x, u)} + (u^*)^\top g(x) + (v^*)^\top h(x) + u^\top g(x^*) + v^\top h(x^*). 
    \end{aligned}
  \]
  Since $u \geq 0$ and $x^* \in \Cc$, we have $g(x^*) \leq 0$ and
  $h(x^*) = 0$, and therefore 
  \[ u^\top g(x^*) + v^\top h(x^*) \leq 0. \]
  Similarly, since $u^* \geq 0$, and $x\in \Cc$, we have $g(x) \leq 0$
  and $h(x) = 0$, and therefore
  \[ (u^*)^\top g(x) + (v^*)^\top h(x) \leq 0.\]
  Finally, since $F$ is
  monotone and $g$ is convex, it follows that $Q(x, u)$ is positive
  semi-definite, and therefore $\Dini_{\Gc_\alpha}V(x) \leq 0$. To
  show (ii), we can use the same reasoning above to show that
  $\Dini_{\Gc_\alpha}V(x) \leq -\mu \norm{x - x^*}^2$.  
\end{proof}

Next, we discuss stability with respect to the entire state space,
which ensures the safe monotone flow can be used to solve
$\VI(F, \Cc)$ even for infeasible initial conditions.

\begin{theorem}\longthmtitle{Stability of Safe Monotone Flow with
    Respect to $\real^n$}
  \label{theorem:stability}
  Assume MFCQ and the constant-rank condition hold for all $x \in \Cc$.
  Then
  \begin{enumerate}
  \item If $x^* \in \SOL(F, \Cc)$ and $F$ is monotone, then $x^*$ is
    globally Lyapunov stable;
  \item If $x^* \in \SOL(F, \Cc)$ and $F$ is $\mu$-strongly monotone,
    then $x^*$ is globally asymptotically stable.
  \end{enumerate}
\end{theorem}

To prove Theorem~\ref{theorem:stability}, we can no longer rely on the
Lyapunov function $V$ defined in \eqref{eq:C-lyap-candidate} because
it is no longer positive definite and may take negative values for
$x \not \in \Cc$. Instead, we consider the new candidate
Lyapunov function
\begin{equation}
  \label{eq:R-lyap-candidate}
  V_\epsilon(x) = \tilde{V}(x) + \bigg[-\frac{1}{\alpha^2} W(x)
  \bigg]_{+} + \delta_\epsilon(x) 
\end{equation}
where $\epsilon > 0$ and $\delta_\epsilon$ is the penalty function
given by
\[
  \delta_\epsilon(x) = \frac{1}{\epsilon}\sum_{i=1}^{m}[g_i(x)]_{+} +
  \frac{1}{\epsilon}\sum_{j=1}^{k}|h_j(x)|.
\]
Before proceeding to the proof of Theorem~\ref{theorem:stability}, we
provide a bound for the Dini derivative of $\delta_\epsilon$
along~$\Gc_\alpha$, which is proven in Appendix \ref{ap:Ddelta}. 

\begin{lemma}[Dini Derivative of $\delta_\epsilon$]
  \label{lem:Ddelta}
  For all $x \in X$ and $\xi \in \real^n$, $\delta_\epsilon$ is
  directionally differentiable at~$x$ along the direction~$\xi$. In particular,
  \begin{equation}
    \label{eq:Ddelta}
    \Dini_{\Gc_\alpha}\delta_\epsilon(x) \leq
    -\frac{\alpha}{\epsilon}\sum_{i \in I_+(x)}g_i(x) -
    \frac{\alpha}{\epsilon}\sum_{j \in I_h(x)} \abs{h_j(x)}, 
  \end{equation}
  where $I_h(x) = \{ j \in [1, k] \mid h_j(x) \neq 0 \}$.
\end{lemma}

We are now ready to prove Theorem~\ref{theorem:stability}.

\begin{proof}[Proof of Theorem~\ref{theorem:stability}]
  We begin by showing (i). Let $x^* \in \SOL(F, \Cc)$. Note that, from
  the optimality conditions \eqref{eq:KKT_proj}, $\Lambda_\alpha(x^*)$
  corresponds to the set of Lagrange multipliers of the solution $x^*$
  to $\VI(F, \Cc)$. Because MFCQ holds at $x^*$, it follows that
  $\Lambda_\alpha(x^*)$ is bounded. Thus, it is possible to choose
  $\epsilon > 0$ small enough so that
  \begin{equation}
    \label{eq:alpha-bound}
    \frac{\alpha}{\epsilon} > \sup_{(u^*, v^*) \in
      \Lambda_\alpha(x^*)} \bigg\{ \norm{\begin{bmatrix}
        u^* \\ v^*
      \end{bmatrix}}_\infty \bigg\}. 
  \end{equation}
  Next, we observe that it follows immediately from the
  definition~\eqref{eq:R-lyap-candidate} that $V_\epsilon$ is positive
  definite with respect to $x^*$. Finally, we compute
  $\Dini_{\Gc_\alpha}V_\epsilon(x)$ and show that it is negative
  semidefinite. Let $x \in X$. We consider three cases: $W(x) < 0$,
  $W(x) > 0$, and $W(x) = 0$.  In the case where $W(x) < 0$,
  \[
    \Dini_{\Gc_\alpha}V_\epsilon = \Dini_{\Gc_\alpha}\tilde{V}(x) -
    \frac{1}{\alpha^2} \Dini_{\Gc_\alpha}W(x) +
    \Dini_{\Gc_\alpha}\delta_\epsilon(x).
  \]
  Combining the bounds in Lemmas~\ref{lem:V1},~\ref{lem:W-properties},
  and~\ref{lem:Ddelta},
  \begin{equation}
    \label{eq:case1}
    \begin{aligned}
        &\Dini_{\Gc_\alpha}V_\epsilon(x) \leq
          -\frac{1}{\alpha^2}\norm{\Gc_\alpha(x)}^2_{Q(x, u)}  +\sum_{i \in I_+(x)}\left(u_i^* -
          \frac{\alpha}{\epsilon}\right)g_i(x) +  \sum_{j \in
          I_h(x)}\left(v_j^* - \frac{\alpha}{\epsilon}\right)\abs{h_j(x)} . 
    \end{aligned}
  \end{equation}
  Since $\epsilon$ satisfies~\eqref{eq:alpha-bound}, it
  follows that $\Dini_{\Gc_\alpha}V_\epsilon(x) \leq 0$.

  For the case where $W(x) > 0$, we rearrange
  \eqref{eq:W-closed-form} to write
  \[
    u^\top g(x) + v^\top h(x) = \frac{1}{\alpha}W(x) +
    \frac{1}{2\alpha}\norm{\Gc(x)}^2 >
    \frac{1}{2\alpha}\norm{\Gc(x)}^2.
  \]
  Then, we have  
  \begin{equation}
    \label{eq:case2}
    \begin{aligned}
      \Dini_{\Gc_\alpha}V_\epsilon(x) &= \Dini_{\Gc_\alpha}\tilde{V}(x) +
        \Dini_{\Gc_\alpha}\delta_\epsilon(x) 
      \\
      &\leq -(u - u^*)^\top g(x) - (v - v^*)^\top h(x) -\frac{\alpha}{\epsilon}\sum_{i \in I_+(x)}g_i(x) -
        \frac{\alpha}{\epsilon}\sum_{j \in I_h(x)} \abs{h_j(x)} 
      \\
      &\leq -\frac{1}{2\alpha}\norm{\Gc_\alpha(x)}^2  + \sum_{i \in I_+(x)}\left(u_i^* -
        \frac{\alpha}{\epsilon}\right)g_i(x) +  \sum_{j \in
        I_h(x)}\left(v_j^* - \frac{\alpha}{\epsilon}\right)\abs{h_j(x)}, 
    \end{aligned}
  \end{equation}
  and $\Dini_{\Gc_\alpha}V_\epsilon(x) \leq 0$.  In the
 case where $W(x) = 0$,
  \[
    \Dini_{\Gc_\alpha}V_\epsilon(x) = \Dini_{\Gc_\alpha}\tilde{V}(x) +
    \frac{1}{\alpha^2} [-\Dini_{\Gc_\alpha}W(x)]_{+} +
    \Dini_{\Gc_\alpha}\delta_\epsilon(x),
  \]
  which leads us to two subcases: (a) $\Dini_{\Gc_\alpha}W(x) < 0$ and
  (b) $\Dini_{\Gc_\alpha}W(x) \geq 0$. In subcase (a),
  $\Dini_{\Gc_\alpha}V_\epsilon(x)$ satisfies the bound in
  \eqref{eq:case1} and, therefore,
  $\Dini_{\Gc_\alpha}V_\epsilon(x) \leq 0$.  In subcase (b),
  \[ 
    u^\top g(x) + v^\top h(x) = \frac{1}{2\alpha}\norm{\Gc_\alpha(x)}^2\] 
  and therefore
  $\Dini_{\Gc_\alpha}V_\epsilon(x)$ satisfies the bound in
  \eqref{eq:case2}, so~$\Dini_{\Gc_\alpha}V_\epsilon(x) \leq 0$.
  To prove~(ii), we can use the same arguments above to show 
  in each case
  $\Dini_{\Gc_\alpha}V_\epsilon(x) \leq -\mu \norm{x - x^*}^2$. 
\end{proof}

We conclude this section by discussing the contraction properties of
the safe monotone flow.  Contraction refers to the property that any
two trajectories of the system approach each other exponentially
(cf. \cite{AD-SJ-FB:22, FB:23} for a precise definition), and implies
exponential stability of an equilibrium. We show that, for
sufficiently large $\alpha$, the safe monotone flow system is
contracting provided $F$ is globally Lipschitz and the constraint set
$\Cc$ is polyhedral. 
Our analysis relies on the following result, which follows as 
a direct consequence of \cite[Lemma 2.1]{NDY:95-MOR}. The proof of this result can be found 
in Appendix \ref{ap:equiv-opt}. 
\begin{lemma}[Linear Program Characterization of Solution to QP]
  \label{lem:qp-to-lp}
  Consider the following quadratic program 
  \begin{equation}
    \label{eq:qp-lemma}
    \underset{(u, v) \in \real^{m}_{\geq 0} \times \real^k}{\textnormal{min}}\;\frac{1}{2}\norm{\begin{bmatrix}
      u \\ v
    \end{bmatrix}}_{\tilde{Q}}^2  + c^\top \begin{bmatrix}
      u \\ v
    \end{bmatrix} + p ,
  \end{equation}
  where $\tilde{Q} \succeq 0$. Then $(u^*, v^*)$ solves
  \eqref{eq:qp-lemma} if and only if it is a solution to the linear
  program
  \begin{equation}
    \label{eq:lp-lemma}
    \underset{(u, v) \in \real^{m}_{\geq 0} \times \real^k}{\textnormal{min}}
    \left(\tilde{Q}\begin{bmatrix}
      u^* \\ v^*
    \end{bmatrix} + c \right)^\top \begin{bmatrix}
      u \\ v
    \end{bmatrix}.
  \end{equation}
\end{lemma}

We now show that the safe monotone flow is contracting. 

\begin{theorem}\longthmtitle{Contraction and Exponential Stability of
  Safe Monotone Flow}\label{thm:contraction}
Let $F$ be $\mu$-strongly monotone and globally Lipschitz with
constant~$\ell_F$ and $\Cc$ a polyhedral set defined by
\eqref{eq:constraint-set} with $g(x) = Gx - c_{g}$ and
$h(x) = Hx - c_{h}$.  
If
\begin{equation}
  \label{eq:contraction-bound}
  \alpha >  \frac{\ell_F^2}{4\mu},
\end{equation}
then the safe monotone flow
is contracting with rate
$ c = \mu - \frac{\ell_F^2}{4\alpha}$.
In particular, the unique solution $x^* \in \SOL(F, \Cc)$ is
globally exponentially stable.
\end{theorem}

\begin{proof}
  We claim that if the assumptions hold, then 
  \begin{equation}
    \label{eq:is-contracting}
    (x - y)^\top (\Gc_\alpha(x) - \Gc_\alpha(y)) \leq -c \norm{x - y}^2, 
  \end{equation}
  in which case the system is contracting by \cite[Theorem
  31]{AD-SJ-FB:22}, and exponential stability of
  $x^* \in \SOL(F, \Cc)$ follows as a consequence.  To show the claim,
  from~\eqref{eq:dlagrangian}, note that
  \[
    \Gc_\alpha(x) = -F(x) - G^\top u_x - H^\top v_x
  \]
  for any $(u_x, v_x) \in \Lambda_\alpha(x)$. Let then $x, y \in X$ and
  $(u_x, v_x) \in \Lambda_\alpha(x)$ and
  $(u_y, v_y) \in \Lambda_\alpha(y)$. Then, using the strong
  monotonicity of $F$,
  \begin{equation}
    \label{eq:contraction-proof}
    \begin{aligned}
      (x - y)^\top (\Gc_\alpha(x) - \Gc_\alpha(y)) &= -(x - y)^\top(F(x) - F(y)) - 
        (x - y)^\top \begin{bmatrix}
         G^\top & H^\top 
        \end{bmatrix}\begin{bmatrix}
         u_x - u_y \\ v_x - v_y
        \end{bmatrix}
      \\
        &\leq -\mu \norm{x - y}^2 - \begin{bmatrix}
            u_x - u_y \\ v_x - v_y
           \end{bmatrix}^\top \begin{bmatrix}
            G(x - y) \\ H(x - y)
           \end{bmatrix}.
    \end{aligned}
  \end{equation}
  Next, let $\tilde{J}(x; u, v) = -\inf_{\xi \in \real^n} L(x; \xi, u, v)$,
  where $L$ is the Lagrangian of \eqref{eq:smooth-proj-grad}, 
  defined
  as
  \begin{align}\label{eq:Lagrangian}
    L(x; \xi, u, v) = \xi^\top F(x) + \frac{1}{2}\norm{\xi}^2 + \sum_{i=1}^{m}u_i(\nabla g_i(x)^\top \xi + \alpha g_i(x)) +
        \sum_{i=1}^{k}v_i(\nabla h_i(x)^\top \xi + \alpha h_i(x)),
  \end{align}
  and $\tilde{Q} \in \real^{(m + k)\times (m + k)}$ is given by 
  \begin{equation}
    \label{eq:tildeQ}
    \tilde{Q} = \begin{bmatrix}
      GG^\top & GH^\top
      \\ 
      HG^\top & HH^\top 
    \end{bmatrix}.
  \end{equation}
  For $x \in \real^n$, $L$ is minimized when
  $\xi = -F(x) - G^\top u - H^\top v$, and therefore
  \begin{equation}
    \label{eq:J-closed-form}
    \begin{aligned}
      \tilde{J}(x; u, v) =
      \frac{1}{2}\norm{
        \begin{bmatrix}
          u \\ v
        \end{bmatrix}}^2_{\tilde{Q}}
        \!\!+ \!
        \begin{bmatrix}
          GF(x) \!-\! \alpha (Gx - c_g)
          \\
          HF(x) \! - \! \alpha (Hx - c_h)
        \end{bmatrix}^\top
        \!\!
        \begin{bmatrix}
        u \\ v
      \end{bmatrix}+ \frac{1}{2}\norm{F(x)}^2. 
    \end{aligned}
  \end{equation}
  If $(u_x, v_x) \in \Lambda_\alpha(x)$, then $(u_x, v_x)$ is a
  solution to the program
  \[\min_{(u, v) \in \real^{m}_{\geq 0} \times \real^{k}} \tilde{J}(x,
  u, v),\] 
  which is the Lagrangian dual\footnote{By convention, the
    Lagrangian dual problem is a maximization problem
    (cf.~\cite[Chapter 5]{SB-LV:09}). However, the minus sign in the
    definition of $\tilde{J}$ ensures that here it is a
    minimization. The reason for this sign convention is to make the
    notation simpler.}  of \eqref{eq:smooth-proj-grad}. 
  By Lemma
  \ref{lem:qp-to-lp}, $(u_x, v_x)$ is also a solution to the linear
  program,
  \[     \underset{(u, v) \in \real^{m}_{\geq 0} \times \real^k}{\textnormal{min}}
  \left(\tilde{Q}\begin{bmatrix}
    u_x \\ v_x
  \end{bmatrix} + \begin{bmatrix} 
    GF(x) - \alpha (Gx - c_g) \\
    HF(x) - \alpha (Hx - c_h)
  \end{bmatrix} \right)^\top \begin{bmatrix}
    u \\ v
  \end{bmatrix}.
  \]
  Since $(u_y, v_y)$ is also feasible for the previous linear program, 
  by optimality of $(u_x, v_x)$,
  \[
    \begin{aligned}
    -&\begin{bmatrix}
      u_x - u_y
      \\
      v_x - v_y
    \end{bmatrix}^\top
    \begin{bmatrix}
      Gx - c_g \\ Hx - c_h
    \end{bmatrix} \leq -\frac{1}{\alpha}\norm{\begin{bmatrix}
      u_x \\ v_x
    \end{bmatrix}}^2_{\tilde{Q}} + \frac{1}{\alpha}\begin{bmatrix}
      u_y \\ v_y
    \end{bmatrix}^\top \tilde{Q}\begin{bmatrix}
      u_x \\ v_x
    \end{bmatrix} -\frac{1}{\alpha}\begin{bmatrix}
      u_x - u_y \\ v_x - v_y
    \end{bmatrix}^\top \begin{bmatrix}
      GF(x) \\ HF(x)
    \end{bmatrix}. 
     \end{aligned}
   \]
   By a similar line of reasoning,
   \[
     \begin{aligned}
    -&\begin{bmatrix}
      u_y - u_x \\ v_y - v_x
    \end{bmatrix}^\top \begin{bmatrix}
      Gy - c_g \\ Hy - c_h
    \end{bmatrix} \leq -\frac{1}{\alpha}\norm{\begin{bmatrix}
      u_y \\ v_y
    \end{bmatrix}}^2_{\tilde{Q}} + \frac{1}{\alpha}\begin{bmatrix}
      u_x \\ v_x
    \end{bmatrix}^\top \tilde{Q}\begin{bmatrix}
      u_y \\ v_y
    \end{bmatrix} -\frac{1}{\alpha}\begin{bmatrix}
      u_y - u_x \\ v_y - v_x
    \end{bmatrix}^\top \begin{bmatrix}
      GF(y) \\ HF(y)
    \end{bmatrix}. 
     \end{aligned}
   \]
  Combining the previous two expressions, we obtain
  \[
    \begin{aligned}
        -\begin{bmatrix}
           u_x - u_y
           \\
           v_x - v_y
         \end{bmatrix}^\top
        \begin{bmatrix}
          G(x - y)
          \\
          H(x - y)
        \end{bmatrix}&\leq -\frac{1}{\alpha}
      \begin{bmatrix}
        u_x - u_y
        \\
        v_x - v_y
      \end{bmatrix}^\top
      \begin{bmatrix}
        G(F(x) - F(y))
        \\
        H(F(x) - F(y))
      \end{bmatrix}
      -\frac{1}{\alpha}
        \norm{
        \begin{bmatrix}
          u_x - u_y
          \\
          v_x - v_y
        \end{bmatrix}}^2_{\tilde{Q}}
      \\
      &\leq \frac{\ell_F}{\alpha}\norm{\begin{bmatrix}
        G \\ H 
      \end{bmatrix}^\top 
        \begin{bmatrix}
          u_x - u_y \\ v_x - v_y
        \end{bmatrix}}
        \norm{x - y} -\frac{1}{\alpha}\norm{\begin{bmatrix}
          G \\ H 
        \end{bmatrix}^\top 
        \begin{bmatrix}
          u_x - u_y
          \\
          v_x - v_y
        \end{bmatrix}}^2 \!\! ,
    \end{aligned}
  \]
  where we used $\|(u, v)\|_{\tilde{Q}} = \norm{[G; H]^\top(u, v)}$.  
  For any $\epsilon > 0$, by Young's
  Inequality~\cite[pp. 140]{HLR-PF:10} it follows that 
  \[
    \begin{aligned}
      -\begin{bmatrix}
          u_x - u_y
          \\
          v_x - v_y
        \end{bmatrix}^\top
        &\begin{bmatrix}
           G(x - y)
           \\
           H(x - y)
         \end{bmatrix}
      \leq \frac{\ell_F}{2\epsilon\alpha}\norm{x - y}^2 -\Big(\frac{1}{\alpha} -
         \frac{\epsilon\ell_F}{2\alpha}\Big) \norm{
          \begin{bmatrix}
            G \\ H
          \end{bmatrix}^\top \begin{bmatrix}
           u_x - u_y
           \\
           v_x - v_y
         \end{bmatrix}}^2 \!\!.
    \end{aligned}
  \]
  Substituting into \eqref{eq:contraction-proof}, we obtain
  \[
    \begin{aligned}
      (x - y)^\top
      &(\Gc_\alpha(x) - \Gc_\alpha(y)) \leq 
        -\Big(\mu - \frac{\ell_F}{2\epsilon\alpha}\Big)
        \norm{x - y}^2 -\Big(\frac{1}{\alpha} -
        \frac{\epsilon\ell_F}{2\alpha}\Big)
        \norm{
          \begin{bmatrix}
            G \\ H
          \end{bmatrix}^\top
        \begin{bmatrix} 
          u_x - u_y
          \\
          v_x - v_y
        \end{bmatrix}}^2.
    \end{aligned}
  \]
  Therefore,~\eqref{eq:is-contracting} holds with
  $c = \mu - \frac{\ell_F}{2\epsilon \alpha}$ if $\epsilon$ satisfies
  \[ \frac{\ell_F}{2 \alpha \mu} \leq \epsilon \leq \frac{2}{\ell_F}. \]
  Such $\epsilon$ can be chosen if $\alpha > \frac{\ell_F^2}{4\mu}$,
  which corresponds to the condition \eqref{eq:contraction-bound}. The
  optimal estimate of the contraction rate is
  $ c = \mu - \frac{\ell_F^2}{4\alpha}$.   
\end{proof}

\begin{remark}[Connection with Safe Gradient Flow]
  \rm{The safe monotone flow is a generalization of the \emph{safe
      gradient flow} introduced in~\cite{AA-JC:24-tac}.  The latter
    was originally studied in the context of nonconvex optimization
    and, similar to the case of the safe monotone flow, enjoys safety
    of the feasible set and correspondence between equilibria and
    critical points. Further, under certain constraint qualifications,
    the local stability of equilibria relative to the constraint set
    under the safe gradient flow can be established using the
    objective function as a Lyapunov function. Because we are working
    with variational inequalities, $F$ may not correspond to the
    gradient of a scalar-valued function, so the Lyapunov functions
    used in~\cite{AA-JC:24-tac} are not directly applicable. However,
    the assumption of convexity and monotonicity here allows us to
    construct novel Lyapunov functions to obtain global stability
    results. \new{The interested reader can check~\cite[Section
      VI]{AA-JC:24-tac} for a comparison of the safe gradient flow
      with other continuous-time methods for optimization.}} \oprocend
\end{remark}

\section{Recursive Safe Monotone Flow}

A drawback of the projected and safe monotone flows is that, in order
to implement them, one needs to solve either the quadratic
programs~\eqref{eq:qp-proj-controller-u-v}
or~\eqref{eq:smf-controller} at each time along the trajectory of the
system.  As a third algorithmic solution to
Problem~\ref{problem:problem}, in this section we introduce the
\emph{recursive safe monotone flow} which gets around this limitation
by incorporating a dynamics whose equilibria correspond to the
solutions of the quadratic program.  
We begin by showing how to derive the dynamics for general constraint
sets $\Cc$ by interconnecting two systems on multiple time scales.
Next, we use the theory of singular perturbations of contracting flows
to obtain stability guarantees in the case where $\Cc$ is polyhedral,
and show that trajectories of the recursive safe monotone flow track
those of the safe monotone flow. The latter property enables us to
formalize a notion of ``practical safety'' that the recursive safe
monotone flow satisfies.

\subsection{Construction of the Dynamics}
We discuss here the construction of the recursive safe monotone flow.
The starting point for our derivation is the control-affine
system~\eqref{eq:variational-system}. The safe monotone flow consists
of this system with a feedback controller specified by the quadratic
program~\eqref{eq:smf-controller}.  Rather than solving this program
exactly, the approach we take is to replace it with a monotone
variational inequality parameterized by the state. For fixed
$x \in X$, we can solve this inequality, and hence obtain the
feedback $\kappa(x)$, using the safe monotone flow corresponding to
this problem. Coupling this flow with the control
system~\eqref{eq:variational-system} yields the~\emph{recursive safe
  monotone flow}. \new{The reason for using the term ``recursive'' is
  that the system can be interpreted as coupling one instance of the
  safe monotone flow (corresponding to the problem
  \eqref{eq:variational-inequality}), with another instance of the
  safe monotone flow (corresponding to the problem~\eqref{eq:qp-cbf-controller}).}

In this section we carry out this strategy in mathematically precise
terms. We rely on the following result, which provides an alternative
characterization of the QP-based
controller~\eqref{eq:smf-controller}.

\begin{lemma}\longthmtitle{Alternative Characterization of Safe Feedback}
  For $x \in \real^n$, consider the optimization
  \begin{equation}
    \label{eq:qp-dual}
    \begin{aligned}
      \underset{(u, v) \in \real_{\geq 0}^m \times
        \real^k}{\textnormal{minimize}} \; &\frac{1}{2}\bigg \|
        \sum_{i=1}^{m}u_i\nabla g_i(x) + 
        \sum_{j=1}^{k}v_j\nabla h_j(x) \bigg \|^2 \\ &+u^\top \bigg(\pfrac{g}{x}F(x) -\alpha g(x)\bigg) + v^\top
        \bigg(\pfrac{h}{x}F(x) -\alpha h(x)\bigg)  .  
    \end{aligned}
  \end{equation}
  If $(u, v)$ is a solution to \eqref{eq:qp-dual}, then $(u, v)$ is a
  solution to \eqref{eq:smf-controller}.
\end{lemma}
\begin{proof}
  Note that the constraints of \eqref{eq:qp-dual} satisfy MFCQ for all
  $(u, v) \in \real^{m}_{\geq 0} \times \real^{k}$. Since the
  objective function in \eqref{eq:qp-dual} is convex in $(u, v)$, one
  can see that necessary and sufficient conditions for optimality are
  given by a KKT system that, after some manipulation, takes the form
  \[
    \begin{aligned}
      -\pfrac{g}{x}\pfrac{g}{x}^\top u -\pfrac{g}{x}\pfrac{h}{x}^\top
      v -\pfrac{g}{x}F(x) -\alpha g(x) &\leq 0
      \\
      -\pfrac{h}{x}\pfrac{g}{x}^\top u - \pfrac{h}{x}\pfrac{h}{x}^\top
      v -\pfrac{h}{x}F(x) -\alpha h(x) &= 0 \\
      u &\geq 0 \\ 
      u^\top \bigg(-\pfrac{g}{x}\pfrac{g}{x}^\top u -\pfrac{g}{x}\pfrac{h}{x}^\top
      v -\pfrac{g}{x}F(x) -\alpha g(x) \bigg) &= 0.
    \end{aligned} 
  \]
  It follows immediately that if $(u, v)$ satisfies the above
  equations, then $(u, v) \in K_{\textnormal{cbf}, \alpha}(x)$. 
\end{proof}

The rationale for considering \eqref{eq:qp-dual}, rather than working
with \eqref{eq:smf-controller} directly, is that the constraints of
\eqref{eq:qp-dual} are independent of~$x$, which will be important for
reasons we show next. Being a constrained optimization problem,
\eqref{eq:qp-dual} can be expressed in terms of a variational
inequality (parameterized by $x \in \real^n$). Formally, let
$\tilde{F}(x, u, v)$ be given by
\begin{align}\label{eq:tildeF}
  \tilde{F}(x, u, v) =
  \begin{bmatrix}
    -\pfrac{g}{x}\Fc(x, u, v) - \alpha g(x)
    \\
    -\pfrac{h}{x}\Fc(x, u, v) -\alpha h(x)
  \end{bmatrix},  
\end{align}
where $\Fc$ is given by~\eqref{eq:variational-system} and let
$\tilde{\Cc} = \real^{m}_{\geq 0} \times \real^{k}$, which we 
parameterize as
\begin{equation}
  \label{eq:tildeC}
  \tilde{\Cc} = \{ (u, v) \in \real^{m} \times \real^{k} \mid u \geq 0 \}.
\end{equation}
The optimization
problem~\eqref{eq:qp-dual} corresponds to the variational inequality
$\VI(\tilde{F}(x, \cdot, \cdot), \tilde{\Cc})$.

Our next step is to write down the safe monotone flow with safety
parameter $\sigma > 0$ corresponding to the variational inequality
$\VI(\tilde{F}(x, \cdot), \tilde{\Cc})$. Note that the
$\sigma$-restricted tangent set~\eqref{eq:approximate-tangent-cone} of
$\tilde{\Cc}$ is
\[
  T^{(\sigma)}_{\tilde{\Cc}}(u, v) =\{
    (\xi_u, \xi_v)
  \in
  \real^{m} \times \real^{k} \;|\; \xi_u \geq -\sigma u\}.
\]
The projection onto $T^{(\sigma)}_{\tilde{\Cc}}(u, v)$ has the
following closed-form solution:
\[
  \proj{T^{(\sigma)}_{\tilde{\Cc}}(u, v)}{
    \begin{bmatrix}
      a \\ b
    \end{bmatrix}
  }= 
  \begin{bmatrix}
    \max\{ -\sigma u, a \}
    \\
    b
  \end{bmatrix}.
\]
Using this expression and applying
Proposition~\ref{prop:smf-equivalent-formulation}, we write the 
safe monotone flow corresponding to 
the variational inequality~$\VI(\tilde{F}(x, \cdot), \tilde{\Cc})$ as
\begin{equation}
  \label{eq:smf-inner}
  \begin{aligned}
    \dot{u} &= \max\bigg\{ -\sigma u, \pfrac{g(x)}{x}\Fc(x, u, v)
    +\alpha g(x) \bigg\} \\
    \dot{v} &= \pfrac{h(x)}{x}\Fc(x, u, v) +\alpha h(x).
  \end{aligned}
\end{equation}
Under certain assumptions, which we formalize in the sequel, for a
fixed $x$, trajectories of~\eqref{eq:smf-inner} converge to solutions
of the~QP~\eqref{eq:smf-controller}.

This discussion suggests a system solving the original variational
inequality $\VI(F, \Cc)$ can be obtained by
coupling~\eqref{eq:smf-inner} with the
dynamics~\eqref{eq:variational-system} as follows:
\begin{subequations}
  \label{eq:recursive-smf}
  \begin{align}
    \dot{x} &= \Fc(x, u, v) \label{eq:slow}
    \\
    \tau\dot{u} &= \max\bigg\{ -\sigma u, \pfrac{g(x)}{x}\Fc(x, u, v)
                  +\alpha g(x) \bigg\} \label{eq:fast1}
    \\ 
    \tau\dot{v} &= \pfrac{h(x)}{x}\Fc(x, u, v)
                  +\alpha h(x). \label{eq:fast2}
  \end{align}
\end{subequations}
We refer to the system \eqref{eq:recursive-smf} as the \emph{recursive
  safe monotone flow}. 
The parameter $\tau$ characterizes the
separation of timescales between the system \eqref{eq:slow} and
\eqref{eq:fast1}-\eqref{eq:fast2}. The interpretation of the dynamics
is that, when~$\tau$ is sufficiently small,
\eqref{eq:fast1}-\eqref{eq:fast2} evolve on a much faster timescale
and rapidly approach the solution set 
of~\eqref{eq:smf-controller}. The system on the slower 
timescale~\eqref{eq:slow} then approximates the safe monotone flow. We formalize
this analysis next.

\new{
  \begin{remark}\longthmtitle{Relationship to Primal-Dual Flows}
    {\rm 
    This recursive safe monotone flow is closely related to
    primal-dual flows, which seek a solution to \eqref{eq:KKT} by
    flowing along a monotone operator in the primal variable, and
    performing a gradient ascent in the dual variables:
    \begin{subequations}
      \label{eq:saddle-flow}
      \begin{align}
        \dot{x} &= -F(x) - \pfrac{g(x)}{x}^\top u -
                  \pfrac{h(x)}{x}^\top v
        \\ 
        \dot{u} &= \proj{T_{\real^{m}_{\geq 0}}(u)}{g(x)}
        \\
        \dot{v} &= h(x)
      \end{align}
    \end{subequations}
    This method has been well-studied
    classically~\cite{KA-LH-HU:58,TK:56}, and has attracted recent
    attention due to its suitability for distributed implementation on
    network problems~\cite{DF-FP:10, AC-BG-JC:17-sicon,
      AC-EM-SHL-JC:18-tac, JC-SKN:19-jnls}.  One drawback of
    \eqref{eq:saddle-flow} is that it does not ensure that the primal
    variable remains in the feasible set $\Cc$. The recursive monotone
    flow can be interpreted as a primal-dual flow which overcomes this
    limitation, and provides practical safety guarantees, by modifying
    the dual-ascent term. One can show that the right-hand side 
    of~\eqref{eq:recursive-smf} converges exactly to the right-hand side
    of~\eqref{eq:saddle-flow} by letting
    $\sigma = \frac{1}{\epsilon^2}$, $\alpha = \frac{1}{\epsilon}$, and
    $\tau = \frac{1}{\epsilon}$, and taking the limit as
    $\epsilon \to 0$. Future work will study further connections of
    the recursive safe monotone flow with primal-dual flows.  }
  \oprocend
  \end{remark}
}

\subsection{Stability of Recursive Safe Monotone Flow}
To prove stability of the system \eqref{eq:recursive-smf}, we rely on
results from contraction theory \cite{AD-SJ-FB:22}. Specifically,
we derive conditions on the timescale separation $\tau$ that ensures
that~\eqref{eq:recursive-smf} is contracting and, as a consequence,
globally attractive and locally exponentially stable.  Throughout the
section, we assume the following assumption holds.

\begin{assumption}\longthmtitle{Strong Monotonicity, Lipschitzness,
    and Polyhedral Constraints}
  \label{as:rmsf-stability}
  The following holds:
  \begin{enumerate}
  \item $F$ is $\mu$-strongly monotone and $\ell_F$-Lipschitz;
  \item $\Cc$ is a polyhedral set defined by \eqref{eq:constraint-set}
    with $g(x) = Gx - c_{g}$ and $h(x) = Hx - c_{h}$, \new{and the
      matrix $[G^\top, H^\top ]^\top$ has full rank\footnote{\new{This
        assumption is slightly stronger than the Linear Independence
        Constraint Qualification (LICQ) assumption. LICQ requires
        linear independence of the gradients of only the active
        constraints, whereas here we require linear independence of
        the gradients of all constraints.}}.
      }
  \end{enumerate}
\end{assumption}

Next, we show that it is possible to tune the parameters 
$\sigma$ so the system
\eqref{eq:smf-inner} is contracting, uniformly in $x$.

\begin{lemma}\longthmtitle{Contractivity of \eqref{eq:smf-inner}}
  \label{lem:inner-contraction}
  Under Assumption~\ref{as:rmsf-stability}, if
  $ 
  \sigma >
    \frac{1}{4}\frac{\lambdamax(\tilde{Q})}{\lambdamin(\tilde{Q})},
  $
  then the system \eqref{eq:smf-inner} is contracting with rate
  $\bar{c} = \lambdamin(\tilde{Q}) - \frac{\lambdamax(\tilde{Q})}{4
    \sigma}$ uniformly in $x$.
\end{lemma}
\begin{proof}
  By Assumption \ref{as:rmsf-stability}, $\tilde{Q} \succ 0$ so one 
  can show $\tilde{F}$ is (i) $\lambdamin(\tilde{Q})$-strongly
  monotone in $(u, v)$ uniformly in $x$ and (ii)
  $\|\tilde{Q}\|$-Lipschitz in $(u, v)$ uniformly in $x$. By
  Theorem~\ref{thm:contraction}, if
  $\sigma > \frac{\| \tilde{Q} \|^2}{4 \lambdamin(\tilde{Q})}$, the
  system \eqref{eq:smf-inner} is uniformly contracting. The result
  follows by observing that
  $\| \tilde{Q} \|^2 = \lambdamax(\tilde{Q})$.   
\end{proof}

We now characterize the contraction and stability properties of the
recursive safe monotone flow.

\begin{theorem}\longthmtitle{Contractivity of Recursive Safe Monotone
    Flow}
  \label{thm:rsmf-contraction}
  Assume $F$ is $\mu$-strongly monotone and $\ell_F$ globally
  Lipschitz, and $\alpha$ satisfies
  \eqref{eq:contraction-bound}. Under
  Assumption~\ref{as:rmsf-stability} and 
  $\sigma$ chosen as in Lemma~\ref{lem:inner-contraction}, then
  \begin{enumerate}
  \item the unique KKT triple, $(x^*, u^*, v^*)$, corresponding to
    $\VI(F, \Cc)$ is the only equilibrium of \eqref{eq:recursive-smf}.
  \end{enumerate}
  Moreover, for all $\epsilon > 0$, there exists $\tau^* > 0$, such
  that for all $0 < \tau < \tau^*$,
  \begin{enumerate}\addtocounter{enumi}{1}
  \item the system \eqref{eq:recursive-smf} is contracting on the set
    \[
      \begin{aligned}
        \Zc_\epsilon = \big\{ (x, u, v) \in
        &X \times \real^{m} \times \real^k \mid \norm{(u, v) - \kappa(x)} \leq \epsilon \big\} ,
      \end{aligned}
    \]
    and every solution of \eqref{eq:recursive-smf} eventually enters
    $\Zc_\epsilon$ in finite time. In particular, there exists a class
    $\Kc\Lc$ function
    $\beta:\real_{\geq 0} \times \real_{\geq 0} \to \real$ such that for
    every solution~$(x(t), u(t), v(t))$
    \[
      \begin{aligned}
        \norm{\big(u(x(t)), v(x(t))\big) - \kappa(x(t))}  \leq \beta \big(\norm{\big(u(x(0)), v(x(0))\big) - \kappa(x(0))},
        t\big) ;
      \end{aligned}
    \]
  \item the unique KKT triple $(x^*, u^*, v^*)$ is locally
    exponentially stable and globally attracting.
  \end{enumerate}
\end{theorem}
\begin{proof}
  We begin with (i). By direct examination
  of~\eqref{eq:recursive-smf}, we see that the equilibria correspond
  exactly with triples satisfying \eqref{eq:KKT}. Since the matrix
  $\tilde{Q}$ has full rank, the gradients of all the constraints 
  are linearly independent,
  and hence MFCQ holds on $\Cc$.  Since $F$ is $\mu$-strongly
  monotone, the solution $x^* \in \SOL(F, \Cc)$ is unique and there
  exists a unique Lagrange multiplier $(u^*, v^*)$ such that
  $(x^*, u^*, v^*)$ satisfies~\eqref{eq:KKT}.
  
  To show (ii), we verify that all hypotheses
  in~\cite[Theorem~4]{LC-FB-EDA:23} hold.  First, note that the map
  $x \mapsto \Fc(x, u, v)$ is $\ell_F$-Lipschitz in $x$ uniformly in
  $(u, v)$, and $\norm{[G; H]}$-Lipschitz in $(u, v)$, uniformly in
  $x$. 
  Let $\Hc$ denote the right-hand side of \eqref{eq:smf-inner}. 
  Because $\Hc$ is piecewise affine in $(u, v)$
  and $F$ globally Lipschitz, there exists constants $\ell_{\Hc, x}$,
  $\ell_{\Hc, u,v} > 0$, such that $\Hc$ is $\ell_{\Hc, x}$-Lipschitz
  in $x$ uniformly in $(u, v)$ and $\ell_{\Hc, u,v}$-Lipschitz in
  $(u, v)$ uniformly in $x$. By Lemma~\ref{lem:inner-contraction},
  there exists $\bar{c} > 0$ such that \eqref{eq:smf-inner} is
  $\bar{c}$-contracting, uniformly in $x$. Finally, we note that the
  reduced system corresponding to \eqref{eq:recursive-smf} is
  $\dot{x} = \Gc_\alpha(x)$, which is contracting by
  Theorem~\ref{thm:contraction}.  Thus, all the hypotheses
  of~\cite[Theorem~4]{LC-FB-EDA:23} hold and (ii) follows.  Finally,
  (iii) follows from combining (i) and (ii).  
\end{proof}

\subsection{Safety of Recursive Safe Monotone Flow}
Here we discuss the safety properties of the recursive safe monotone
flow.  In general, even if $x(0) \in \Cc$, 
it is not guaranteed that solutions of the
system \eqref{eq:recursive-smf} satisfy $x(t) \in \Cc$ for $t > 0$.
However, under appropriate conditions, we can show that the system is
``practically safe'', in the sense that $x(t)$ remains in a slightly
expanded form of the original constraint set $\Cc$.

\begin{theorem}\longthmtitle{Practical Safety of Recursive Safe
    Monotone Flow}\label{thm:approx-safety}
  Assume $F$ is $\mu$-strongly monotone and $\ell_F$ globally
  Lipschitz, and $\alpha$ satisfies
  \eqref{eq:contraction-bound}. Under
  Assumption~\ref{as:rmsf-stability} and 
  $\sigma$ chosen as in Lemma~\ref{lem:inner-contraction}, then for all
  $\epsilon > 0$, there exists $\delta > 0$ and $\tau^*$ such that, if
  $0< \tau < \tau^*$, any solution to~\eqref{eq:recursive-smf} with
  \begin{itemize}
    \item $x(0) \in \Cc$; 
    \item $\norm{(u(0), v(0)) - \kappa(x(0))} \leq \delta$;
  \end{itemize}
  satisfies
  $x(t) \in \Cc_\epsilon$ for all $t \geq 0$, where 
  $\Cc_\epsilon  =  \{ x \in \real^n  \mid g(x) \leq \epsilon, \abs{h(x)} \leq \epsilon \}$.
\end{theorem}

To prove Theorem~\ref{thm:approx-safety}, we rely on the notion of
input-to-state safety. Consider the system
\begin{equation}
  \label{eq:perturbed-smf}
  \dot{x} = \Gc_\alpha(x) - \sum_{i=1}^{n}e_i(t)\nabla g_i(x) -
  \sum_{j=1}^{m}d_j(t)\nabla h_j(x). 
\end{equation}

This system can be interpreted as the safe monotone flow perturbed by
a disturbance~$(e(t), d(t))$. The set~$\Cc$ is
\emph{input-to-state safe} (ISSf) with respect to
\eqref{eq:perturbed-smf}, with gain $\gamma$, if there exists a class
$\Kc$ function $\gamma$ such that, if
$\gamma(\norm{(e, d)}_\infty) < \epsilon$, then $\Cc_\epsilon$ is
forward invariant under~\eqref{eq:perturbed-smf}. This notion of
input-to-state safety is a slight generalization of the standard
definition, cf.~\cite{SK-ADA:18}, to the case where the safe set is
parameterized by multiple equality and inequality constraints.  We
show next that~\eqref{eq:perturbed-smf} is ISSf.

\begin{lemma}[Perturbed Safe Monotone Flow is ISSf]
  \label{lem:issf}
  Under Assumption~\ref{as:rmsf-stability}, the set $\Cc$
  is input-to-state safe with respect to \eqref{eq:perturbed-smf} with
  gain $\gamma(r) = \frac{\lambdamax(\tilde{Q})}{\alpha} r$, where
  $\tilde{Q}$ is defined in \eqref{eq:tildeQ}
\end{lemma}
\begin{proof}
  For $i \in \until{m}$, under~\eqref{eq:perturbed-smf}
  \[
    \begin{aligned}
      \dot{g}_i(x) & = G_i^\top \bigg(\Gc_\alpha(x) - \sum_{i=1}^{n}e_i(t)\nabla g_i(x) -
        \sum_{j=1}^{m}d_j(t)\nabla h_j(x) \bigg) \notag
      \\
      & \leq -\alpha g_i(x) - G^\top_i \bigg(\sum_{i=1}^{n}e_i(t)\nabla g_i(x) +
        \sum_{j=1}^{m}d_j(t)\nabla h_j(x)\bigg) \notag
      \\
      & \leq -\alpha g_i(x) + \lambdamax(\tilde{Q}) \norm{(e(t), d(t))} ,
    \end{aligned}
  \]
  where $G^\top_i$ is the $i$th row of $G$.  It follows from
  \cite[Theorem~1]{SK-ADA:18} that the set
  $\Cc_{g_i} = \{x \in \real^n \mid G_i^\top x - (c_g)_i \leq 0 \}$ is
  input-to-state safe with gain $\gamma$
  with respect to~\eqref{eq:recursive-smf}.

  For $j \in \until{k}$, under~\eqref{eq:perturbed-smf},
  \[
    \begin{aligned}
      \dot{h}
      &_j(x) = H_j^\top \bigg(\Gc_\alpha(x) - \sum_{i=1}^{n}e_i(t)\nabla g_i(x) -
        \sum_{j=1}^{m}d_j(t)\nabla h_j(x) \bigg)
      \\
      &= -\alpha h_j(x) - H_j^\top \bigg(\sum_{i=1}^{n}e_i(t)\nabla g_i(x) +
        \sum_{j=1}^{m}d_j(t)\nabla h_j(x) \bigg) ,
    \end{aligned}
  \]
  where $H^\top_j$ is the $j$th row of $H$. It follows that
  \[
    \begin{aligned}
      \dot{h}_j(x)
      &\leq -\alpha h_j(x) +
        \lambdamax(\tilde{Q})\norm{(e(t), d(t))} ,
      \\ 
      \dot{h}_j(x)
      &\geq -\alpha h_j(x) - \lambdamax(\tilde{Q})\norm{(e(t), d(t))} .
    \end{aligned}
  \]
  Thus, by \cite[Theorem~1]{SK-ADA:18}, the sets
  $\Cc^{-}_{h_j} = \{x \in \real^n \mid H_j^\top x - (c_h)_j \leq 0
  \}$, and
  $\Cc^{+}_{h_j} = \{x \in \real^n \mid H_j^\top x - (c_h)_j \geq 0
  \}$ are also input-to-state safe with gain $\gamma$ with respect to
  \eqref{eq:recursive-smf}.  Finally, input-to-state safety of $\Cc$
  follows from the fact that
  \[
    \Cc = \bigg(\bigcap_{i = 1}^{m}\Cc_{g_i}\bigg) \cap \bigg(\bigcap_{j
      = 1}^{k}(\Cc^{+}_{h_j} \cap \Cc^{-}_{h_j})\bigg).  
  \] 
\end{proof}

We are now ready to prove Theorem~\ref{thm:approx-safety}.

\begin{proof}[Proof of Theorem~\ref{thm:approx-safety}]
  By Lemma~\ref{lem:issf}, $\Cc$ is input-to-state safe with respect
  to \eqref{eq:perturbed-smf}, with gain
  $\gamma(r) = \frac{\lambdamax(\tilde{Q})}{\alpha} r$.  Note that, for any solution
  $(x(t), u(t), v(t))$ of \eqref{eq:recursive-smf}, the trajectory
  $x(t)$ solves \eqref{eq:perturbed-smf} with
  \[
    \begin{bmatrix}
      e(t)
      \\
      d(t)
    \end{bmatrix}
    =
    \begin{bmatrix}
      u(t)
      \\
      v(t)
    \end{bmatrix}
    - \kappa(x(t)).
  \]
  Next, by Theorem~\ref{thm:rsmf-contraction}, for all $\epsilon$,
  there exists $\tau^* > 0$ such that if $0 < \tau < \tau^*$, then for
  all $t \geq 0$,
  $
    \norm{\big( e(t), d(t) \big)} \leq \beta \big( \norm{ \big(e(0),
      d(0) \big) }, t \big)
  $
  for some class $\Kc\Lc$ function $\beta$. Now, choose $\delta > 0$
  such that $\alpha^{-1}\lambdamax(\tilde{Q})\beta(\delta, 0) < \epsilon$ and let
  $\norm{(u(0), v(0)) - \kappa(x(0))} \leq \delta$.  Then, for all
  $t \geq 0$,
  \[
    \begin{aligned}
      \gamma \left( \norm{\big( e(t), d(t) \big)} \right)
      &\leq
        \gamma(\beta(\delta,
        t)) \leq
        \gamma(\beta(\delta,
        0)) <
        \epsilon .
    \end{aligned}
  \]
  Hence, for $x(0) \in \Cc \subset \Cc_\epsilon$, since $\Cc$ is
  input-to-state safe with respect to \eqref{eq:perturbed-smf}, we
  conclude $x(t) \in \Cc_\epsilon$ for all~$t \geq 0$.  
\end{proof}

\begin{remark}[Implementation of Recursive Safe Monotone Flow 
  on Network Optimization Problem] {\rm One drawback of the
  projected monotone flow and the safe monotone flow is that neither
  system admits a distributed implementation for network
  optimization problems.  The reason for this is that both the
  feedback control and projected implementations of these systems
  introduce additional couplings
  between agents that require global knowledge to solve exactly.

  We show here that the recursive safe monotone flow can be
  implemented in a distributed manner on constrained optimization
  problems with a separable objective function and locally expressible
  constraints, provided that agents can communicate with their two-hop
  neighbors. Consider an undirected network $(\Vc, \Ec)$ where
  $\Vc = \{1, 2, \dots, N\}$ is a set of agents, and $\Ec$ is the set
  of edges.  For each agent $i \in \Vc$, let $\Nc_i$ denote the set of
  neighbors of $i$.  Consider
  \begin{equation}\label{eq:opt}
    \begin{aligned} 
      &\underset{x \in \real^{n}}{\textnormal{minimize}} &&f(x) =
      \sum_{i \in \Vc}f_i(x_i)
      \\
      &\textnormal{subject to} &&\sum_{j \in \Nc_i}G_{ij}x_j \leq c_{g_i},
      \;\;\forall i \in \Vc
    \end{aligned}
  \end{equation}
  where $x_i$ corresponds to the decision variable of the $i$th agent. 
  For problems with this structure, 
  the constraints corresponding to each 
  node are expressible as a function of information available to that node. 
  This assumption is common in the distributed optimization literature
  (see e.g. \cite{SB-NP-EC-BP-JE:11}). 
  When implementing \eqref{eq:recursive-smf}, 
  the state of the $i$th agent consists of $(x_i, u_i)$, and 
  the dynamics are given by 
  \[
    \begin{aligned}
      \dot{x_i} &= -\nabla f_i(x_i) - \sum_{j \in \Nc_i}G_{ij}^\top u_j \\
      \tau\dot{u_i} &= \text{max}\bigg\{ -\sigma u_i, \sum_{j \in \Nc_i}G_{ij} \bigg( -\nabla f_j(x_j) - \sum_{k \in \Nc_j} G_{jk}^\top u_k  \bigg) + \alpha \bigg( \sum_{j \in \Nc_i}G_{ij}x_j - c_{g_i} \bigg) \bigg\}
    \end{aligned} 
  \]
  Observe that from the previous expression, 
  the $i$th node can implement the dynamics corresponding to its state using 
  only the state information of its two-hop neighbors.  }  \oprocend
\end{remark}

\section{Numerical Examples}
Here we illustrate the behavior of the proposed flows on two example
problems.  The first example is a variational inequality on $\real^2$
corresponding to a two-player game with quadratic payoff functions.
The second example is a constrained linear-quadratic dynamic game
where we implement the safe monotone flow in a receding horizon manner
to examine its anytime properties.

\subsection{Nash Equilibria of Two-Player Game}
The first numerical example we discuss is a variational inequality on
$\real^2$ corresponding to a two-player game, where player
$i \in \{1, 2 \}$ wants to minimize a cost $J_i(x_1, x_2)$ subject to
the constraints that $x_i \in \Cc_i \subset \real$.  We take
$\Cc = \Cc_1 \times \Cc_2 \subset \real^2$.  We have selected a
two-dimensional example that allows us to visualize the constraint set
and the trajectories of the proposed flows to better illustrate their
differences.
The problem of finding the Nash equilibria of a game of this form is
equivalent to the variational inequality $\VI(F, \Cc)$, where $F$ is
the \emph{pseudogradient} map, given by
$ F(x) = ( \nabla_{x_1}J_1(x_1, x_2), \nabla_{x_2}J_2(x_1, x_2)) $.
For $i \in \{1, 2 \}$, let
$\Cc_i = \{ x \in \real \mid -1 \leq x \leq 1 \}$ and $J_i$ be the
quadratic function
$ J_i(x_1, x_2) = \frac{1}{2}x^\top Q_i x + r_i^\top x$, with
\[
  \begin{aligned}
    &Q_1 =
      \begin{bmatrix}
        1 & -1 \\ -1 & 1
      \end{bmatrix}, \qquad
      &Q_2 = \begin{bmatrix}
        1 & 1 \\ 1 & 1
      \end{bmatrix},
      \qquad
    &&r_1 =
       \begin{bmatrix}
         0 \\ 0 
       \end{bmatrix}
      \qquad
    &&r_2
       = \begin{bmatrix}
           0.5 \\ 0.5
         \end{bmatrix}.
  \end{aligned}
\]
The pseudogradient map is given by $F(x) = Qx + r$ where 
\[ 
Q = \begin{bmatrix}
  1 & -1 \\ 1 & 1
\end{bmatrix}  \qquad r = \begin{bmatrix}
  0 \\ 0.5
\end{bmatrix}
\]
Because $\frac{1}{2}(Q + Q^\top) = I \succ 0$, it 
follows that the $F$ is $1$-strongly monotone, and
therefore the problem has a unique solution $x^* \in \SOL(F, \Cc)$.

Figure \ref{fig:toy-example} shows the results of implementing each of
the proposed flows to find the Nash equilibrium.  The projected
monotone flow, cf. Figure~\ref{fig:toy-example}(a), is only 
well-defined in $\Cc$. However, the constraint set remains forward
invariant and all trajectories converge to the solution $x^*$.  The
safe monotone flow with $\alpha = 1.0$, cf.
Figure~\ref{fig:toy-example}(b), also keeps the constraint set forward
invariant and has all trajectories converge to $x^*$. In addition, the
system is well-defined outside of $\Cc$, and trajectories beginning
outside the feasible set converge to it.

In Figure~\ref{fig:toy-example}(c), we consider the recursive safe
monotone flow with $\alpha = 1.0$, $\sigma = 1.0$ and $\tau = 0.25$,
where $u(0) = 0$. The trajectories converge to $x^*$ and closely
approximate the trajectories of the safe monotone flow.
Note, however, that the set $\Cc$ is not safe but only practically
safe. This is illustrated in the zoomed-in
Figure~\ref{fig:toy-example}(d), where it is apparent that the
trajectories do not always remain in~$\Cc$ but remain close to~it.

\begin{figure}[!htb]
  \centering \subfigure[Projected monotone flow]{
    \includegraphics[width=0.44\linewidth]{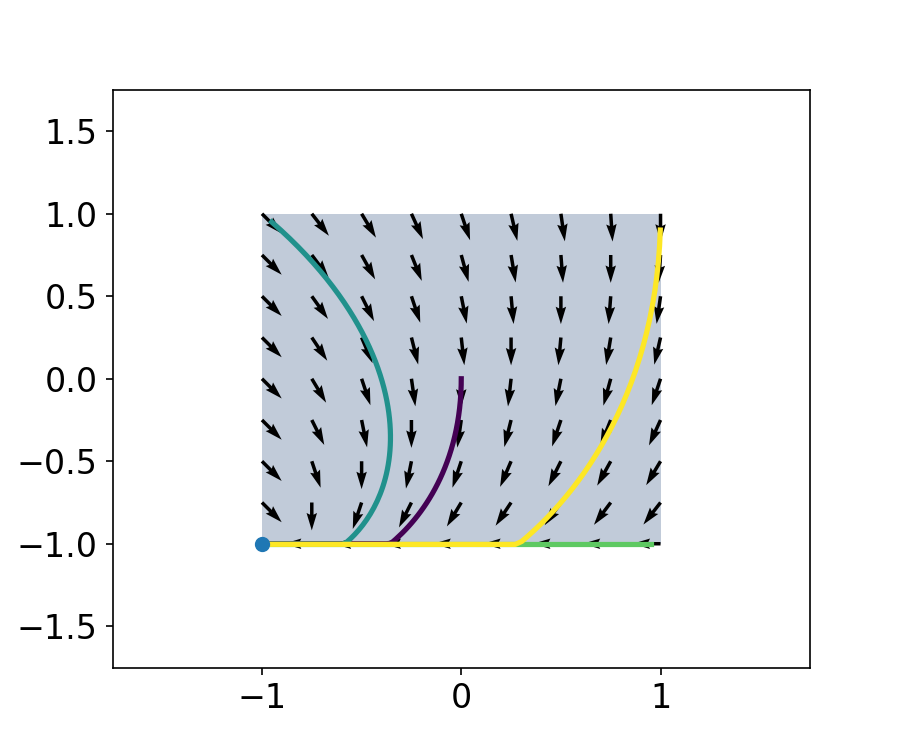}}
  \subfigure[Safe monotone flow]{
    \includegraphics[width=0.44\linewidth]{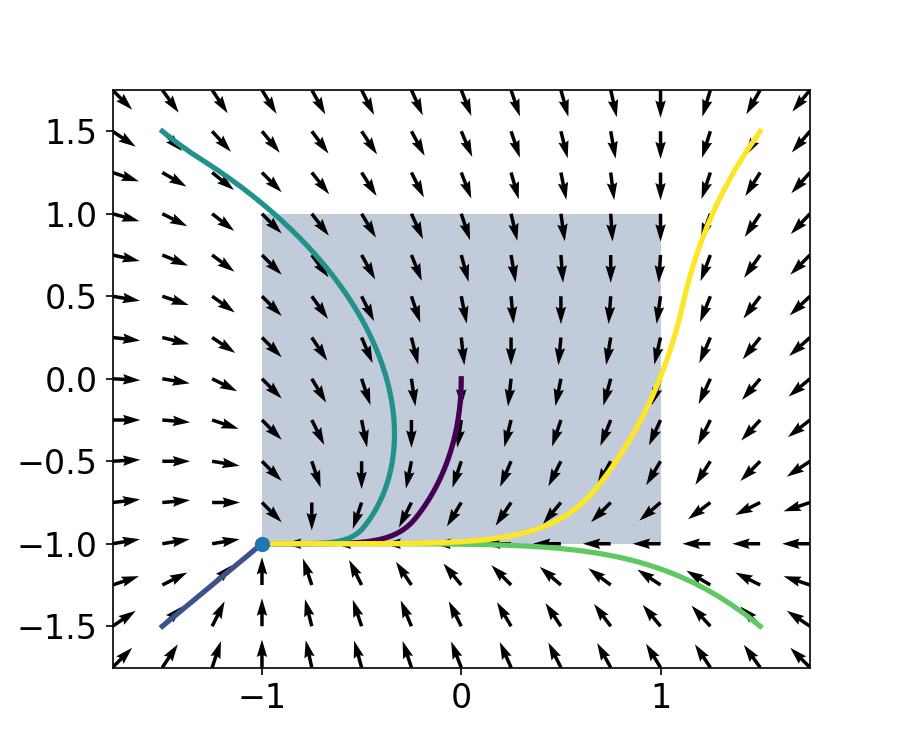}}
  \\
  \subfigure[Recursive safe monotone flow]{
    \includegraphics[width=0.44\linewidth]{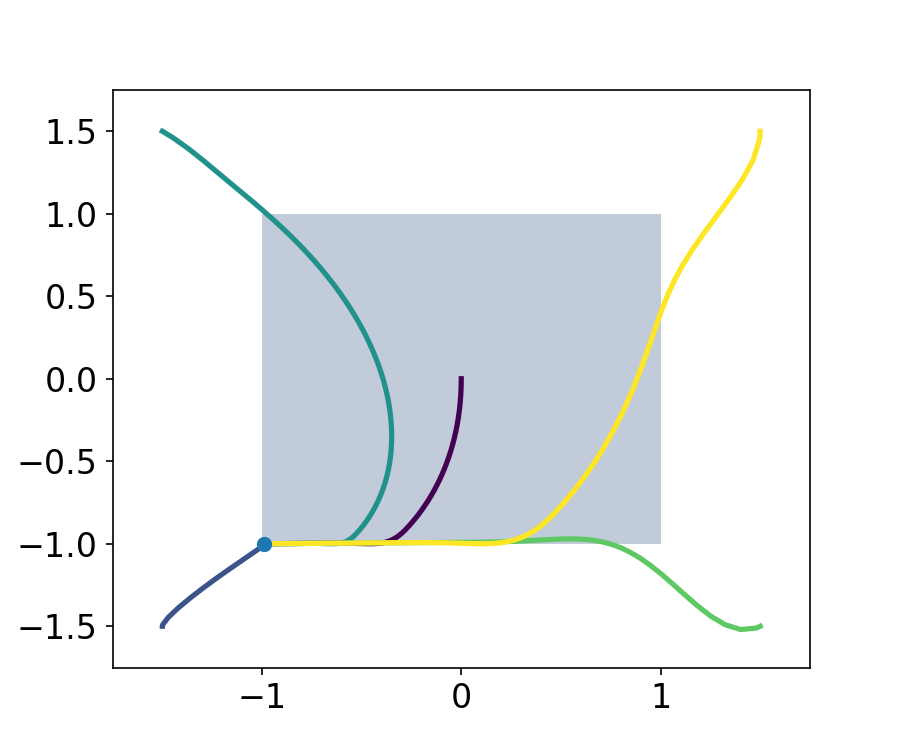}}
  \subfigure[Zoomed-in plot of (c)]{
    \includegraphics[width=0.44\linewidth]{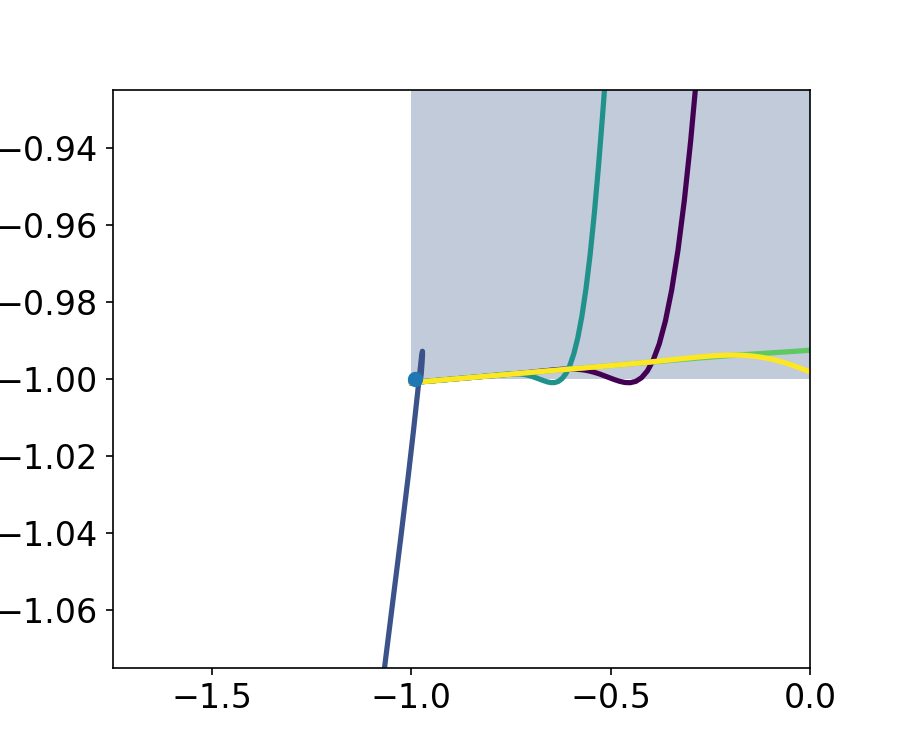}}
  \caption{Implementation of (a) projected monotone flow, (b) safe
    monotone flow ($\alpha=1.0$), and (c) recursive safe monotone flow
    ($\tau=0.25$) in a two-player game. The shaded region shows the
    constraint set $\Cc$ and the colored paths represent trajectories
    of the corresponding flow starting from various initial
    condition. (d) shows a zoomed-in portion of the boundary of $\Cc$
    to illustrate the practical safety of the recursive safe monotone
    flow.  }\label{fig:toy-example}
\end{figure}

\subsection{Receding Horizon Linear-Quadratic Dynamic Game}
We now discuss a more complex example, where the input to a plant is
specified by the solution to a variational inequality parameterized by
the state of the plant. To solve it, we interconnect the plant
dynamics with the safe monotone flow, and demonstrate that the anytime
property of the latter ensures good performance and satisfaction of
the constraints even when terminated early.  The plant
takes the form of a discrete-time linear time-invariant system with
two inputs,
\begin{equation}
  \label{eq:LTI}
  z(s + 1) = Az(s) + B_1w_1(s) + B_2w_2(s) ,
\end{equation}
where $A \in \real^{n_z \times n_z}$ and
$B_i \in \real^{n_z \times n_w}$ for $i \in \{1, 2\}$. 

We consider a
linear-quadratic dynamic game (LQDG) between two players, where each
player $i \in \{1, 2\}$ can influence the system \eqref{eq:LTI} by
choosing the corresponding input $w_i \in W_i \subset \real^{n_w}$.
We fix a time horizon $N > 0$ and an initial condition $z(0) = z_0$,
and define $J:(W_1)^N \times (W_2)^N \to \real$ as a quadratic payoff,
\begin{equation}
  \label{eq:LQDG-cost}
  \begin{aligned}
    J(w_1(\cdot),
    &w_2(\cdot)) = \norm{z(N)}^2_{Q_f} + \sum_{s=0}^{N - 1}\norm{z(s)}^2_{Q} + \norm{w_1(s)}_{R_1}^2 -
      \norm{w_2(s)}_{R_2}^2, 
  \end{aligned} 
\end{equation}
where $Q_f, Q \succeq 0$ and $R_1, R_2 \succ 0$.  The goal of player 1
is to minimize the payoff~\eqref{eq:LQDG-cost}, whereas the goal of
player 2 is to maximize it. This problem can be solved in closed form
when the constraints $W_i$ are trivial (cf. \cite[Chapter
6]{TB-GJO:99}, \cite{MP-KDP:10}), but must be solved numerically for
nontrivial constraints.

We first note the LQDG problem can be posed as a variational
inequality.  Indeed, let $\bar{z} = (z(1), \dots, z(N))$ and, for
$i \in \{1, 2\}$, let $\bar{w}_i = (w_i(0), \dots, w_i(N - 1))$. 
Define
\[
  \Ac =
  \begin{bmatrix}
    A \\ A^2 \\ \vdots \\ A^{N}
  \end{bmatrix},
  \qquad C_i =
  \begin{bmatrix}
    B_i & 0  & \cdots & 0 \\ 
    AB_i & B_i &  \cdots & 0 \\ 
    \vdots & \vdots & \ddots & \vdots \\ 
    A^{N - 1}B_i & A^{N - 2}B_i   & \cdots & B_i
  \end{bmatrix}.
\]
Next, letting $\bar{Q} = \text{diag}(Q, \dots, Q, Q_f)$
and $\bar{R}_i = \text{diag}(R_i, \dots, R_i)$,
and using the fact that
\[ \bar{z} = \Ac z_0 + C_1\bar{w}_1 + C_2\bar{w}_2,\]
we see that the
payoff function \eqref{eq:LQDG-cost} can be written as,
\[
  \begin{aligned}
    J(\bar{w}_1, \bar{w}_2) =
    &\begin{bmatrix}
       \bar{w}_1
       \\
       \bar{w}_2
     \end{bmatrix}^\top
      \begin{bmatrix}
        C_1^\top \bar{Q} C_1 + \bar{R}_1 & C_1^\top \bar{Q} C_2
        \\ 
        C_2^\top \bar{Q} C_1 & C_2^\top \bar{Q}C_2 - \bar{R}_2
      \end{bmatrix}
      \begin{bmatrix}
        \bar{w}_1
        \\
        \bar{w}_2
      \end{bmatrix} +2
      \begin{bmatrix}
        C_1^\top \bar{Q}\Ac z_0
        \\
    C_2^\top \bar{Q}\Ac z_0
      \end{bmatrix}^\top
      \begin{bmatrix}
        \bar{w}_1
        \\
        \bar{w}_2
  \end{bmatrix} + z_0^\top \Ac^\top \bar{Q}\Ac z_0.
  \end{aligned}
\]
Finally, letting $x = (\bar{w}_1, \bar{w}_2)$, we see that the LQDG problem
corresponds to the variational inequality $\VI(F(\cdot, z_0), \Cc)$, where
\begin{align*}
  F(x, z_0) 
  \! =\!
    \begin{bmatrix}
      C_1^\top \bar{Q} C_1 + \bar{R}_1 & C_1^\top \bar{Q} C_2
      \\
      -C_2^\top \bar{Q} C_1 & \bar{R}_2 - C_2^\top \bar{Q}C_2
    \end{bmatrix}
  \!
    \begin{bmatrix}
      \bar{w}_1
      \\ \bar{w}_2
    \end{bmatrix}
\!    +\!
    \begin{bmatrix}
      C_1^\top \bar{Q}\Ac z_0
      \\
      -C_2^\top \bar{Q}\Ac z_0
    \end{bmatrix}  
\end{align*}
and the constraint set is $\Cc = (W_1)^N \times (W_2)^N$. 
We assume that the problem data satisfies
\[
  \begin{bmatrix}
    C_1^\top \bar{Q} C_1 + \bar{R}_1 & 0
    \\
    0 & \bar{R}_2 - C_2^\top \bar{Q}C_2
  \end{bmatrix} \succ 0,
\]
in which case $F$ is strongly monotone.

For simulation purposes, we take $n_z = 5$, $n_w = 2$, $B_i = I$,
$W_i = \real^{2}_{\geq 0}$, and $A$ a marginally stable matrix
selected randomly.  We use the safe monotone flow to solve the
variational inequality and implement the solution in a receding
horizon manner: given the initial state $z_0$, we solve for the
optimal input sequence $(w_1(\cdot), w_2(\cdot))$ over the entire time
horizon, apply the input $(w_1(0), w_2(0))$ to \eqref{eq:LTI} to
obtain $z(1)$, update the initial condition $z_0 \leftarrow z(1)$ and
repeat. When $F$ is strongly monotone, on each iteration the flow
converges to the exact solution as $t \to \infty$. However, we also
consider here the effect of terminating the solver early at
some~$t = t_f < \infty$.

Figure~\ref{fig:lqdg-example} shows the results of the simulation.  In
Figure~\ref{fig:lqdg-example}(a), we plot $\norm{z(s)}$ for various
values of termination times. We denote the exact solution with
$t_f = \infty$.  The closed-loop dynamics with the exact solution to
the receding horizon LQDG is stabilizing, and as $t_f$ grows larger,
the early terminated solution drives the state of the system closer to
the origin.  In Figure \ref{fig:lqdg-example}(b), we plot the first
component of $w_1(s)$ in blue and the first component of $w_2(s)$ in
red.  Regardless of when terminated, the inputs satisfy the input
constraints on each iteration due to the safety properties of the safe
monotone flow.

\begin{figure}[!htb]
  \centering
  \subfigure[$\norm{z(s)}$ on each iteration]{
    \includegraphics[width=0.40\linewidth]{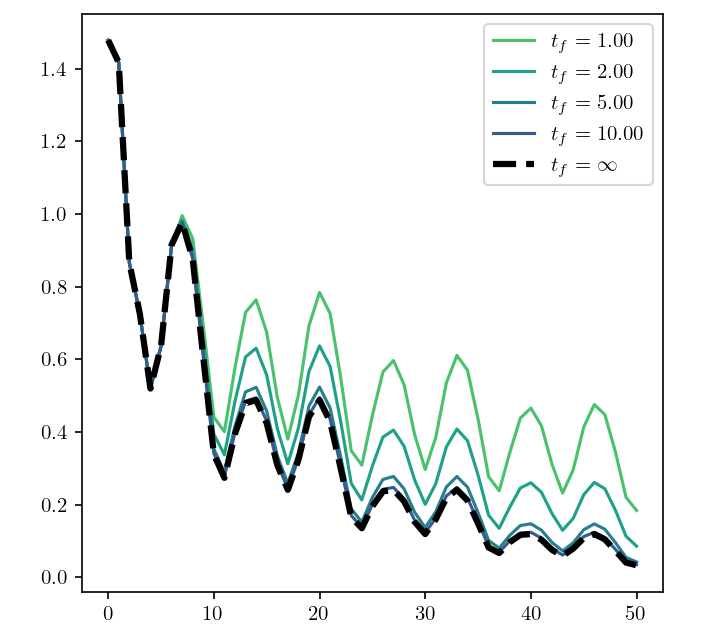}
  }
  \hspace{4ex}
  \subfigure[First component of $\{w_i(s)\}_{i \in \{1,2\}}$ on each
  iteration]{ 
    \includegraphics[width=0.40\linewidth]{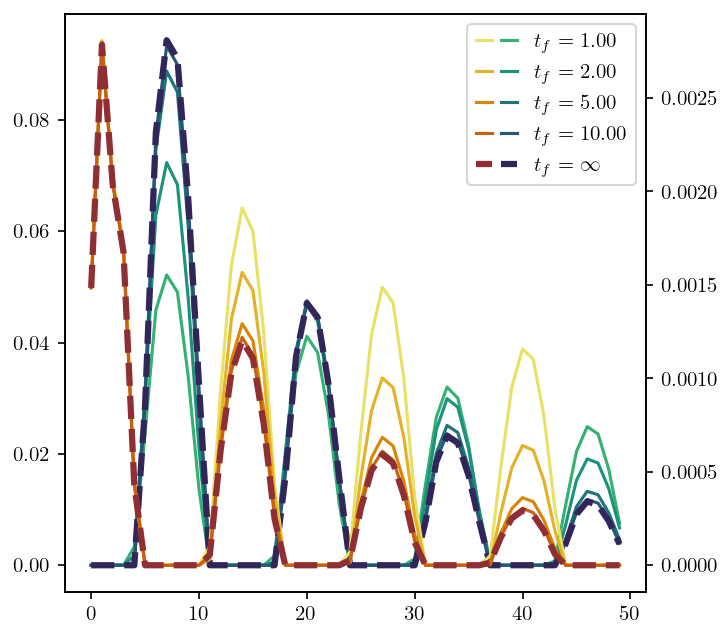}
  }
  \caption{Receding horizon implementation of the safe monotone flow
    solving a linear quadratic dynamic game for different choices of
    termination time $t_f$. The closed-loop implementation of the
    exact solution corresponds to $t_f = \infty$ (dashed lines). (a)
    We plot the evolution of $\norm{z(s)}$ in green. (b) We plot the
    evolution of the first component of $w_1(s)$ in blue-green (scale
    in left $y$-axis) and the first component of $w_2(x)$ in
    red-orange (scale in right $y$-axis). }
  \label{fig:lqdg-example}
\end{figure}

\section{Conclusions}
We have tackled the design of anytime algorithms to solve variational
inequalities as a feedback control problem.  Using techniques from
safety-critical control, we have synthesized three continuous-time
dynamics which find solutions to monotone variational inequalities:
the projected monotone flow, already well known in the literature, and
the novel safe monotone and recursive safe monotone flows.  The
equilibria of these flows correspond to solutions of the variational
inequality, and so we have embarked in the precise characterization of
their asymptotic stability properties. We have established asymptotic
stability of equilibria in the case of strong monotonicity, and
contractivity and exponential stability in the case of polyhedral
constraints. We have also shown that the safe monotone flow renders
the constraint forward invariant and asymptotically stable.  The
recursive safe monotone flow offers an alternative implementation that
does not necessitate the solution of a quadratic program along the
trajectories. This flow results from coupling two systems evolving on
different timescales, and we have established local exponential
stability and global attractivity of equilibria, as well as practical
safety guarantees.  We have illustrated in two game scenarios the
properties of the proposed flows and, in particular, their amenability
for interconnection and regulation of physical processes.  Future work
will develop methods for distributed network problems, \new{study
  time-varying and parametric problems}, consider applications to
feedback optimization, \new{explore connections between the recursive
  safe monotone flow and primal-dual flows}, and \new{analyze
  discretizations and the implementation on computational platforms of
  the proposed flows.}

\appendix
\section{Appendix: Proof of Lemma \ref{lem:V1}}
\label{ap:V1}
\begin{proof}
  Note that
  \[
    \begin{aligned}
      & \Dini_{\Gc_\alpha}\tilde{V}(x)
        = -(x - x^*)^\top F(x)  -\sum_{i=1}^{m}u_i(x - x^*)^\top\nabla g_i(x) -
        \sum_{j=1}^{k}v_j(x - x^*)^\top\nabla h_j(x).   
    \end{aligned}
  \]
  When $F$ is $\mu$-strongly monotone,
  \[-(x - x^*)^\top F(x) \leq -\mu\norm{x - x^*}^2 - (x - x^*)^\top
  F(x^*),\]
  and when $F$ is monotone, the previous inequality holds with $\mu=0$.
  Next, we rearrange \eqref{eq:KKT1} and use that $g_i$ is convex for
  all $i=1, \dots, m$ and $h_j$ is affine for all $j=1, \dots, m$ to
  obtain
  \[
    \begin{aligned}
      - (x - x^*)^\top F(x^*) &=\sum_{i=1}^{m}u_i^*(x - x^*)^\top\nabla g_i(x^*) +
      \sum_{j=1}^{k}v_j^*(x - x^*)^\top\nabla h_j(x^*)
      \\
      &\leq \sum_{i=1}^{m}u_i^*(g_i(x) - g_i(x^*)) +
      \sum_{j=1}^{k}v_j^*(h_j(x) - h_j(x^*)) \\
      &= (u^*)^\top (g(x) - g(x^*)) + (v^*)^\top h(x), 
    \end{aligned} 
  \]
  where the last equality follows from the fact that $h(x^*) = 0$. 
  By a similar line of reasoning, we have 
  \[
    \begin{aligned}
      -\sum_{i=1}^{m}&u_i(x - x^*)^\top\nabla g_i(x) -
        \sum_{j=1}^{k}v_j(x - x^*)^\top\nabla h_j(x)
      \\
      &\leq -\sum_{i=1}^{m}u_i(g_i(x) - g_i(x^*)) -
        \sum_{j=1}^{k}v_j(h_j(x) - h_j(x^*))
      \\
      &= -u^\top (g(x) - g(x^*)) - v^\top h(x).
    \end{aligned}
  \]
  The result follows by summing the previous two expressions.   
\end{proof}

\section{Appendix: Proof of Lemma \ref{lem:W-properties}}
\label{ap:W-properties}
\begin{proof}
  We first show that the solution to the optimization problem in the
  definition of $W$ is $\xi = \Gc_\alpha(x)$. Note that the
  constraints in the definition of $W$ in~\eqref{eq:C-lyap-candidate}
  and \eqref{eq:smooth-proj-grad} are identical. Let $J(x, \xi)$
  denote the objective function in the definition of $W$. Then
  $ J(x, \xi) - \frac{1}{2}\norm{\xi + F(x)}^2 =
  -\frac{1}{2}\norm{F(x)}^2$. Because the difference between the
  objective functions of \eqref{eq:smooth-proj-grad} and the
  definition of $W$ is independent of $\xi$, the two optimization
  problems have the same solution.  The claim now follows because $\Gc_\alpha(x)$ is 
  the minimizer solving \eqref{eq:smooth-proj-grad}. 

  Next we show that $W$ can be expressed in closed form as
  \eqref{eq:W-closed-form}.  Because the optimizer is
  $\xi = \Gc_\alpha(x)$, we have
  \begin{equation}
    \label{eq:W-inter}
    W(x) = \Gc_\alpha(x)^\top F(x) + \frac{1}{2}\norm{\Gc_\alpha(x)}^2.
  \end{equation} 
  Note that $(\Gc_\alpha(x), u, v)$ satisfies the optimality
  conditions \eqref{eq:KKT_proj} for all
  $(u, v) \in \Lambda_\alpha(x)$. Therefore, we can rearrange
  \eqref{eq:dlagrangian} to obtain
  $F(x) = -\Gc_\alpha(x) - \pfrac{g(x)}{x}^\top u -
  \pfrac{h(x)}{x}^\top v$. Next
  \[ \begin{aligned}
       \Gc_\alpha(x)^\top F(x) &= -\norm{\Gc_\alpha(x)}^2 - u^\top \pfrac{g(x)}{x}\Gc_\alpha(x) -
      v^\top \pfrac{h(x)}{x}\Gc_\alpha(x)
       \\
    &=-\norm{\Gc_\alpha(x)}^2 + \alpha u^\top g(x) + \alpha v^\top h(x),
     \end{aligned} \]
   where the second equality follows by rearranging 
   \eqref{eq:pequality} and \eqref{eq:transversality}. 
   Then, \eqref{eq:W-closed-form} follows by substituting the previous
   expression into~\eqref{eq:W-inter}.  

   Finally, we show~\eqref{eq:DW}. By \cite[Theorem~4.2]{VB-AL-ML:16}, it follows that 
    \[ \begin{aligned}
      \Dini_{\Gc_\alpha}W(x) &= \sup_{(u', v') \in \Lambda_\alpha(x)}
        \left\{\nabla_x L(x; \Gc_\alpha(x), u',
        v')^\top\Gc_\alpha(x)\right\}\\
        & \geq \nabla_x L(x; \Gc_\alpha(x), u, v)^\top\Gc_\alpha(x),
      \end{aligned} \] 
  where $L$ is defined in \eqref{eq:Lagrangian}. 
    By direct computation, we can verify that 
    \[
      \begin{aligned} 
        \nabla_x L(x; \xi, u, v)= Q(x, u)\xi +
          \alpha\pfrac{g(x)}{x}^\top u +
          \alpha\pfrac{h(x)}{x}^\top v 
      \end{aligned}
    \]
    Therefore 
    \[
      \begin{aligned}
        \nabla_x L(x; \Gc_\alpha(x),u, v)^\top\Gc_\alpha(x) &=
        \norm{\Gc_\alpha(x)}^2_{Q(x, u)} + \alpha u^\top \pfrac{g(x)}{x}\Gc_\alpha(x) + \alpha
          v^\top \pfrac{h(x)}{x} \Gc_\alpha(x) \\
          &= \norm{\Gc_\alpha(x)}^2_{Q(x, u)} - \alpha^2 u^\top
          g(x) - \alpha^2 v^\top h(x), 
      \end{aligned} 
    \]
    where the last equality above follows by
    rearranging~\eqref{eq:pequality} and \eqref{eq:transversality}. 
\end{proof}

\section{Appendix: Proof of Lemma \ref{lem:Ddelta}}
\label{ap:Ddelta}

\begin{proof}
  Note that $\delta_\epsilon$ corresponds to the $\ell^1$-penalty
  function for the set $\Cc$. By \cite[Proposition 3]{GdP-LG:89}, the
  directional derivative of $\delta_\epsilon$ is
  \[
    \begin{aligned}
      \delta'_\epsilon
      (x; \xi) =\; &\frac{1}{\epsilon}\sum_{i \in I_+(x)}\nabla g_i(x)^\top \xi 
        + \frac{1}{\epsilon}\sum_{i \in I_0(x)} [\nabla g_i(x)^\top  \xi ]_+ \\
      &+\frac{1}{\epsilon}\sum_{j \in I_h(x)}\sgn(h_j(x))\nabla h_j(x)^\top  \xi 
        + \frac{1}{\epsilon}\sum_{j \not\in I_h(x)}\abs{\nabla h_j(x)^\top \xi }.
    \end{aligned}
  \]
  Note
  $\Dini_{\Gc_\alpha}\delta_\epsilon(x) = \delta'_\epsilon(x;
  \Gc_\alpha(x))$.  Then~\eqref{eq:Ddelta} follows by noting
  that $\nabla g_i(x)^\top \Gc_\alpha(x) \leq -\alpha g_i(x)$ 
  for all~$i=1, \dots, m$ and
  $\nabla h_j(x)^\top \Gc_\alpha(x) = -\alpha h_j(x)$ for all~$j=1,
  \dots k$.  
\end{proof}

\section{Appendix: Proof of Lemma \ref{lem:qp-to-lp}}
\label{ap:equiv-opt}
The proof of Lemma \ref{lem:qp-to-lp} relies on the following lemma. 

\begin{lemma}
  \label{lem:equiv-opt}
  Let $\Uc \subset \real^{p}$ and $\Vc \subset \real^{q}$, 
  and consider the optimization problems 
  \begin{align}
    \label{eq:equiv-opt-1}
    &\min_{y \in \Uc} J(y), \\
    \label{eq:equiv-opt-2}
    &\min_{z \in \Vc} \bar{J}(z),
  \end{align}
  where $J:\Uc \to \real$ and $\bar{J}:\Vc \to \real$. 
  Suppose there exists a map $\varphi:\Uc \to \Vc$
  such that
  \begin{enumerate}
    \item $\bar{J} = J \circ \varphi$; 
    \item $\varphi$ is surjective;
  \end{enumerate}
  Then $z^* \in \Vc$ is a solution to \eqref{eq:equiv-opt-2} if and only if $y^* = \varphi(z^*) \in \Uc$
  is a solution to \eqref{eq:equiv-opt-1}. 
\end{lemma}

\begin{proof}
  We begin with the forward direction. Suppose that $z^* \in \Vc$ is a solution to \eqref{eq:equiv-opt-2}, 
  and let $y^* = \varphi(z^*)$. Then $J(y^*) = \bar{J}(z^*)$ and for all $y \in \Uc$, there exists $z_y \in \Vc$ 
  such that $\varphi(z_y) = y$, and
  \[ J(y) = \bar{J}(z_y) \geq \bar{J}(z^*) = J(y^*). \]
  Hence, $y^*$ solves \eqref{eq:equiv-opt-1}. To prove the reverse direction, suppose that $y^* \in \Uc$ solves \eqref{eq:equiv-opt-1} and
  choose $z^* \in \Vc$ such that $y^* = \varphi(z^*)$. Next, for all $z \in \Vc$,
  \[ \bar{J}(z) = J(\varphi(z)) \geq J(y^*) = \bar{J}(z^*), \] 
  so $z^*$ must be a solution to \eqref{eq:equiv-opt-2}. 
\end{proof}

We now proceed with the proof of Lemma \ref{lem:qp-to-lp}. 
\begin{proof}
  Consider the quadratic program \eqref{eq:qp-lemma}, and partition $\tilde{Q} \in \real^{(m + k)} \times \real^{(m + k)}$
  and $c \in \real^{m} \times \real^{k}$ as
  \[ \tilde{Q} =  \begin{bmatrix}
    \tilde{Q}_{11} & \tilde{Q}_{12} \\
    \tilde{Q}_{21} & \tilde{Q}_{22}
  \end{bmatrix} \qquad c = \begin{bmatrix}
    c_1 \\ c_2
  \end{bmatrix}. \] 
  Next, define the following quadratic program 
  \begin{equation}
    \label{eq:qp-lemma-aug}
    \underset{(u, v_+, v_-) \in \real^{m}_{\geq 0} \times \real^{k}_{\geq 0} \times \real^{k}_{\geq 0}}{\textnormal{min}}\;
    \frac{1}{2}\begin{bmatrix}
      u \\ v_+ \\ v_-
    \end{bmatrix}^\top \begin{bmatrix}
      \tilde{Q}_{11} & \tilde{Q}_{12} & -\tilde{Q}_{12} \\
      \tilde{Q}_{21} & \tilde{Q}_{22} & -\tilde{Q}_{22} \\ 
      -\tilde{Q}_{21} & -\tilde{Q}_{22} & \tilde{Q}_{22}
    \end{bmatrix} \begin{bmatrix}
      u \\ v_+ \\ v_-
    \end{bmatrix}  + \begin{bmatrix} c_1 \\ c_2 \\ -c_2 \end{bmatrix}^\top \begin{bmatrix}
      u \\ v_+ \\ v_-
    \end{bmatrix} + p ,
  \end{equation}
  and the following linear program 
  \begin{equation}
    \label{eq:lp-lemma-aug}
    \underset{(u, v_+, v_-) \in \real^{m}_{\geq 0} \times \real^{k}_{\geq 0} \times \real^{k}_{\geq 0}}{\textnormal{min}}\;
    \left(\begin{bmatrix}
      \tilde{Q}_{11} & \tilde{Q}_{12} & -\tilde{Q}_{12} \\
      \tilde{Q}_{21} & \tilde{Q}_{22} & -\tilde{Q}_{22} \\ 
      -\tilde{Q}_{21} & -\tilde{Q}_{22} & \tilde{Q}_{22}
    \end{bmatrix}\begin{bmatrix}
      u^* \\ v_+^* \\ v_-^*
    \end{bmatrix} + \begin{bmatrix} c_1 \\ c_2 \\ -c_2 \end{bmatrix} \right)^\top \begin{bmatrix}
      u \\ v_+ \\ v_-
    \end{bmatrix},
  \end{equation}
  where $v_+^* \in \real^{k}_{\geq 0}$ and $v_-^* \in \real^{k}_{\geq 0}$ 
  are any vectors such 
  that $v^* = v_+^* - v_-^*$. By applying Lemma \ref{lem:equiv-opt} with the surjection 
  $\varphi:\real^{m}_{\geq 0} \times \real^{k}_{\geq 0} \times \real^{k}_{\geq 0} \to \real^{m}_{\geq 0} \times \real^{k}$, 
  where $\varphi(u, v_+, v_-) = (u, v_+ - v_-)$, we see that $(u^*, v^*)$ is a solution to \eqref{eq:qp-lemma} 
  if and only if $(u^*, v_+^*, v_-^*)$ is a solution to \eqref{eq:qp-lemma-aug}. Next 
  by \cite[Lemma 2.1]{NDY:95-MOR}, $(u^*, v_+^*, v_-^*)$ is a solution to \eqref{eq:qp-lemma-aug} 
  if and only if $(u^*, v_+^*, v_-^*)$ is a solution to \eqref{eq:lp-lemma-aug}. Finally, by applying 
  Lemma \ref{lem:equiv-opt} once more, we see that $(u^*, v_+^*, v_-^*)$ is a solution to \eqref{eq:lp-lemma-aug} if 
  and only if $(u^*, v^*)$ is a solution to \eqref{eq:lp-lemma}, and the result follows.  
\end{proof}

\end{document}